\documentclass[10pt,a4paper,twoside]{article}
\usepackage {amsmath}
\usepackage {amsfonts}
\usepackage {amssymb}
\usepackage {amsthm}
\usepackage {latexsym}
\usepackage {graphics}
\usepackage {graphicx}
\usepackage {fancyhdr}
\usepackage {color}
\newcommand {\be}{\begin{enumerate}}
\newcommand {\ee}{\end{enumerate}}

\newcommand {\Z} {\mathbb Z}

\newcommand {\startenv} {\begin{quote}\begin{tabular}{||l}\parbox[t]{12cm}}
\newcommand {\stopenv} {\end{tabular}\end{quote}}

\newcommand {\TV} {\mathrm{TV}}
\newtheorem{theorem}{Theorem}[section]
\newtheorem{corollary}{Corollary}[section]
\newtheorem{definition}{Definition}[section]
\newtheorem*{definition*}{Definition}
\newtheorem{lemma}{Lemma}[section]

\newtheorem{remark}{Remark}
\newtheorem{requirement}{Requirement}
\newtheorem*{requirement*}{Requirement}

\begin{document}
\title{Adaptive Mesh Reconstruction \\ Total Variation Bound}
\author{Nikolaos Sfakianakis\footnote{University of Vienna, 1090 Vienna, Austria}}
\maketitle
\bibliographystyle{amsalpha}
\begin{abstract}
We consider 3-point numerical schemes for scalar Conservation Laws, that are oscillatory either to their dispersive or anti-diffusive nature. Oscillations are responsible for the increase of the Total Variation (TV); a bound on which is crucial for the stability of the numerical scheme. It has been noticed (\cite{Arvanitis.2001}, \cite{Arvanitis.2004}, \cite{Sfakianakis.2008}) that the use of non-uniform adaptively redefined meshes, that take into account the geometry of the numerical solution itself, is capable of taming oscillations; hence improving the stability properties of the numerical schemes. 

In this work we provide a model for studying the evolution of the extremes over non-uniform adaptively redefined meshes. Based on this model we prove that proper mesh reconstruction is able to control the oscillations; we provide bounds for the Total Variation (TV) of the numerical solution. We moreover prove under more strict assumptions that the increase of the TV -due to the oscillatory behaviour of the numerical schemes- decreases with time; hence proving that the overall scheme is TV Increase-Decreasing (TVI-D).
\end{abstract}
\section{Outline}\label{Section.Outline}
We study the scalar Conservation Law in one space variable
$$u_t+f(u)_x=0,\quad x\in[0,1],\ t\in[0,T]$$
where the flux function $f$ is considered to be smooth and convex. The initial condition throughout this work is:
$$u_0(x)=\begin{cases} 1,& x\leq x_0\\0, &x>x_0 \end{cases},\quad\mbox{ for }x_0\in(0,1)$$

We discretise spatially over a non-uniform, adaptively redefined mesh, 
$$M^n=\left\{0=x_0^n<x_1^n<\cdots <x_N^n=1\right\}$$

The manipulation of the non-uniform mesh and the time evolution are combined into the Main Adaptive Scheme (MAS): 
\begin{definition*}[Main Adaptive Scheme (MAS)]
  Given, at time step $n$, the mesh $M_x^n=\{a=x_1^n<\cdots<x_N^n=b\}$ and the approximations $U^n=\{u_1^n,\ldots,u_N^n\}$, the steps of the (MAS) are as 
  follows:
  \begin{itemize}
    \item[1.](Mesh Reconstruction)\\
      Construct new mesh $M_x^{n+1}=\{a=x_1^{n+1}<\cdots<x_N^{n+1}=b\}$
    \item[2.](Solution Update)\\
      Using the old mesh $M_x^n$ the approximations $U^n$ and the new mesh $M_x^{n+1}$:
      \begin{itemize}
         \item[2a.] construct a piecewise linear function $V^n(x)$ such that $V^n|_{M_x^n}=U^n$ 
         \item[2b.] define the updated approximations $\hat U^n=\{\hat u_i^n,\ldots,\hat u_N^n\}$ as $\hat U^n= V^n|_{M_x^{n+1}}$
      \end{itemize}
    \item[3.](Time Evolution)\\
      Use the new mesh $M_x^{n+1}$, the new approximations $\hat U^n$ and the numerical scheme to march in time and compute 
      $U^{n+1}=\{u_1^{n+1},\ldots,u_N^{n+1}\}$
    \item[4.](Loop)\\
      Repeat the Step 1.-3. with $M_x^{n+1}$, $U^{n+1}$ as initial data.
  \end{itemize}
\end{definition*}
We shall discuss the MAS in more details in section \ref{Section.MAS}, for the moment we shall make some brief notes on the Steps 1.-3.
\begin{remark}
  Regarding the mesh reconstruction (Step 1.), we note that it is performed in each time step and is the crucial ingredient of the MAS. This procedure creates a 
  new non-uniform mesh $M_x^{n+1}$ with the same number of nodes as the previous one $M_x^n$. The construction of the  new mesh depends on the geometry of the 
  numerical solution itself. Regarding the solution update (Step 2.) and time evolution (Step 3.), we note that in this work we consider interpolation over piecewise 
  linear functions and Finite Difference schemes. The proper setting for conservative solution update and Finite Volume schemes will be addressed in a different work.
\end{remark}

The objective of this work is to place conditions on the steps of MAS such that the resulting numerical solutions are TV stable even when oscillatory numerical schemes are used for the time evolution (Step 3.).  More specifically, we shall prove that proper non-uniform meshes are able to tame the TV increase due to oscillations, furthermore we shall prove that in some cases the TV increase due to oscillation decreases with time; hence yielding a Total Variation Increase diminishing (TVID) scheme. 


In section \ref{Section.Requirements} of this work we list and explain the requirements that we place on the MAS. In section \ref{Section.Analysis}, we discuss the creation and propagation of oscillations at the level of extremes. We present a model for the extremes that takes into account the steps of MAS. Based then on the model we prove the TV results of this work. In section \ref{Section.MAS} we discuss the Main Adaptive Scheme (MAS) in more detail. We analyse the way non-uniform meshes are constructed and how the numerical solution is updated over the new mesh. We discuss the numerical implementation of the MAS and present some graphs depicting its basic properties. In section \ref{Section.Considerations} we discuss the numerical implementation of the requirements introduced in section \ref{Section.Requirements}. In section \ref{Section.Numerics} we perform numerical tests, where we consider known oscillatory numerical schemes and prove that these schemes satisfy the requirements introduced in section \ref{Section.Requirements}. We provide comparative numerical results for both uniform and non-uniform meshes.

\section{Requirements}\label{Section.Requirements}


In this section we present the requirements that we place on the steps of the MAS. We once again mention that we perform interpolation over piecewise linear functions for the Solution Update (Step 2.) and that we use oscillatory Finite Difference schemes for the Time Evolution (Step 3.). The proper setting for dealing with Finite Volume schemes with a conservative reconstruction shares many things in common with this work but will be presented separately. 

Let $M_x^n=\{x_i^n,i=1,\ldots,N\}$ be the initial mesh, $U^n=\{u_1^n,\ldots, u_N^n\}$ the initial approximations at the time step $n$, and $M_x^{n+1}=\{x_j^{n+1},j=1,\ldots,N\}$, $\hat U^n=\{\hat u_1^n,\ldots,\hat u_N^n\}$ be the new mesh and updated approximations yielding from the Steps 1. and 2. of the MAS.

To introduce the first requirement, we recall that the development of this work does not assume the use of any specific numerical scheme for time evolution (Step 3.). In the contrary the discussion that will take place and the proofs that will be presented are valid for every numerical scheme that satisfies the \emph{Evolution requirement}:

\begin{requirement}[Evolution requirement]\label{EvolutionReq}
  There exists a constant $C>0$ independent of the time step $n$ and the node $i$ such that,
  \begin{equation}
    |u_i^{n+1}-\hat u_i^n|\leq C\max\left\{|\hat u_{i+1}^n-\hat u_i^n|,|\hat u_i^n-\hat u_{i-1}^n|\right\}
  \end{equation}
\end{requirement}

\begin{remark}
  The meaning of this requirement is that the numerical scheme, which is responsible for the time evolution (Step 3.), does not produce abrupt results. We shall see 
  in the examples addressed in section \ref{Section.Numerics} that the constant $C$ is an increasing function of the CFL condition. For every scheme that we will use, we shall prove this 
  requirement and also compute the dependence of the constant $C$ to the $CFL$ condition.
\end{remark}

We move on to the second requirement, which is placed on the mesh reconstruction procedure (Step 1).  We first start with a definition,
\begin{definition}
  We say that the approximate solution $U^n=\{u_i^n,i=1,\ldots,N\}$ defined over the mesh $M_x^n=\{x_i^n,i=1,\ldots,N\}$ exhibits a local extreme at the node $x_i^n$ if $u_i^n>u_{i-1}^n,u_{i+1}^n$ (local maximum) or $u_i^n<u_{i-1}^n,u_{i+1}^n$ (local minimum).
\end{definition}

\begin{requirement}[Mesh requirement: $\lambda$-rule]\label{l-ruleReq}
  There exists a constant $0<\lambda<1$ independent of $n$ and $i$ such that, if $x_j^{n+1}\in[x_i^n,x_{i+1}^n]$ and $U^n$ exhibits a local extreme at the node $i$ 
  (resp. $i+1$) then
  \begin{equation}\label{l-ruleEq1}
    x_j^{n+1}-x_i^n > (1-\lambda) (x_{i+1}^n-x_i^n)\quad \Big( \mbox{respectively}\quad x_{i+1}^n-x_j^{n+1} > (1-\lambda) (x_{i+1}^n-x_i^n)\Big)
  \end{equation}
  or equivalently
  \begin{equation}\label{l-ruleEq2}
    x_{i+1}^n-x_j^{n+1} < \lambda (x_{i+1}^n-x_i^n)\quad \Big( \mbox{respectively}\quad x_j^{n+1} -x_i^n < \lambda (x_{i+1}^n-x_i^n)\Big)
  \end{equation}

\end{requirement}
\begin{figure}[t]
  \begin{center}  \input{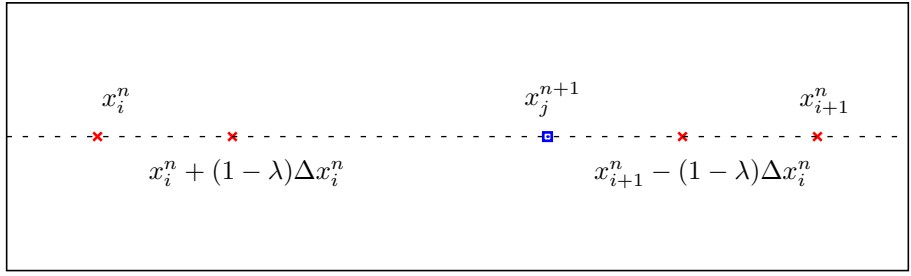}  \end{center}
    \caption{The $\lambda$-rule states that the \emph{new} node $x_{j}^{n+1}\in[x_i^n,x_{i+1}^n]$ should avoid the \emph{old} node $x_i^n$ (respectively $x_{i+1}^n$), 
             where the approximate solution might exhibit an extreme, in the sense: $x_{i}^n+(1-\lambda)\Delta x_{i}^n\leq x_{j}^{n+1}$ 
             (respectively $x_{j}^{n+1}\leq x_{i+1}^n-(1-\lambda)\Delta x_{i}^n$), where $\Delta x_{i}^n=x_{i+1}^n-x_{i}^n$
            }\label{Graph.l-ruleMesh}
\end{figure}
\begin{remark}  
  The meaning of the $\lambda$-rule requirement, Rel.(\ref{l-ruleEq1}), is that the new nodes $x_j^{n+1}$ avoid the places of the old extremes $x_i^n$ (or $x_{i+1}^n$) 
  by at least $1-\lambda$ of the respective interval $(x_i^n,x_{i+1}^n)$.
\end{remark}
The $\lambda$-rule Requirement is placed on the mesh reconstruction but also affects the values of piecewise linear functions, we call this $\lambda$-rule effect. The following remark discusses their relation,
\begin{remark}[$\lambda$-rule effect for piecewise linears and interpolation]\label{Rem.l-rule}
  Assume that $u(x)$ is a piecewise linear function that oscillates as depicted in Fig.(\ref{l-ruleGeneral}). Assume moreover that the new nodes respect the 
  $\lambda$-rule Req.(\ref{l-ruleReq}) at the extremes in the respective subintervals. Let also $y=\upsilon$ be a horizontal line that separates the extremes. 
  \begin{figure}
    \input{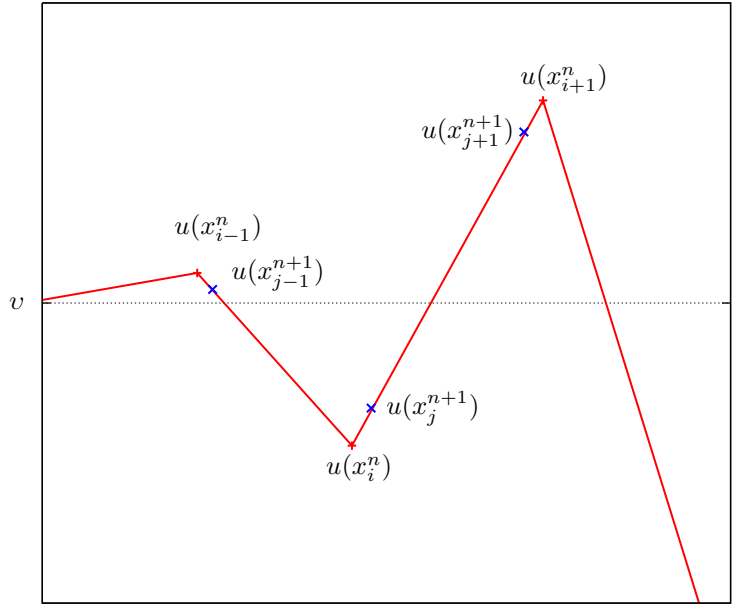}
     \caption{This figure depicts the application of the $\lambda$-rule in the case of a piecewise linear function. The places of the new nodes are depicted along 
              with the old extremes. }\label{l-ruleGeneral}
  \end{figure}
  According to Fig.(\ref{l-ruleGeneral}), $x_{j-1}^{n+1}\in[x_{i-1}^n,x_i^n]$ and by the $\lambda$-rule requirement 
  $$x_{j-1}^{n+1}-x_{i-1}^n\geq (1-\lambda)(x_i^n-x_{i-1}^n)$$
  we get
  $$\frac{x_i^n-x_{j-1}^{n+1}}{x_i^n-x_{i-1}^n}\leq\lambda$$
  Since $u$ is linear in the interval $[x_{i-1}^n,x_i^n]$
  $$\frac{u(x_i^n)-u(x_{j-1}^{n+1})}{u(x_i^n)-u(x_{i-1}^n)}=\frac{x_i^n-x_{j-1}^{n+1}}{x_i^n-x_{i-1}^n}\leq\lambda$$
  by the monotonicity of $u$ in the interval $[x_{i-1}^n,x_i^n]$ the previous relation recasts into
  $$u(x_{j-1}^{n+1})-u(x_i^n)\leq \lambda(u(x_{i-1}^n)-u(x_i^n))$$
  since $0<\lambda<1$ and $u(x_i^n)<\upsilon$ the previous relation reads 
  $$u(x_{j-1}^{n+1})\leq \lambda u(x_{i-1}^n)+(1-\lambda)u(x_i^n)\leq \lambda u(x_{i-1}^n)+(1-\lambda)\upsilon\Rightarrow 
    u(x_{j-1}^{n+1})-\upsilon\leq \lambda(u(x_{i-1}^n)-\upsilon)$$

  For the rest of the extremes of Fig.(\ref{l-ruleGeneral}), that is  $u(x_i^n)$ and $u(x_{i+1}^n)$ and for the respective new nodes $x_j^{n+1}$, $x_{j+1}^{n+1}$ we 
  can similarly prove that,
  $$|u(x_j^{n+1})-\upsilon|\leq\lambda |u(x_i^n)-\upsilon|\quad\mbox{ and }\quad|u(x_{j+1}^{n+1})-\upsilon|\leq \lambda|u(x_{i+1}^n)-\upsilon|$$
  The meaning of this remark is the following: if a new node respects the $\lambda$-rule in a specific interval, the same is true for the value of this new node 
  with respect to the values of the function at the end points of this interval.
\end{remark}

The case where several nodes are placed between consequent extremes will be addressed in Section \ref{Section.Considerations}. There 

The last remark connects the $\lambda$-rule -which is applied on the mesh- with the values of the piecewise linear function. In the rest of this work we shall refer to this connection as \emph{the $\lambda$-rule effect}.

To pass to the third requirement, we first note that the overall phenomenon consists of two major steps, the mesh reconstruction (Step 1.) and the time evolution (Step 3.). We need to study these steps both separately and together. For the separate analysis the requirements Req.(\ref{EvolutionReq}) and Req.(\ref{l-ruleReq}) are sufficient, but for the joined analysis one more requirement is needed. The necessity for the extra requirement is explained in the following remark.
\begin{remark}
  The effect of time evolution (Step 3.) which is due to the numerical scheme, can be studied with several means, such as the respective Modified Equation of the 
  scheme. In the contrary, the effect of the mesh reconstruction (Step 1.) cannot be analysed with classical methods, since it takes place between two consequent time 
  steps. In other words, the mesh reconstruction is not related to the time evolution of the numerical solution.
\end{remark}

The major contribution of this work is the merging of the effects that these steps have on the appearance and the evolution of oscillations; hence on the TV of the numerical approximation. The merging of these effects is quantified in the \emph{Coupling requirement}:
\begin{requirement}[Coupling requirement]\label{CouplingReq}
  The constants $C$ of the Evolution Req.(\ref{EvolutionReq}) and $\lambda$ of the $\lambda$-rule Req.(\ref{l-ruleReq}) are connected via the following relation:
  \begin{equation}
    \lambda+3\lambda C<1
  \end{equation}
\end{requirement}

\section{Time step analysis}\label{Section.Analysis}
  In this section we discuss the appearance and evolution of local extremes. We devise recursive, with respect to the time step $n$, relations regarding the evolution 
  of the extremes based on the requirements presented in the previous section. We present and prove the theoretical results of this work; these include bounds on the 
  extremes, bounds on the TV increase due to oscillations, we moreover prove that in some cases the TV increase, decreases with time steps $n$.
\subsection{Recursive relations}
  The discussion regarding the creation and evolution of the extremes is performed in a time step by time step manner. In every time step we shall discuss their 
  temporal evolution and their spatial modification.

  We start with a jump initial condition, Fig.(\ref{Graph.Initial.Condition}), which we discretize over a non-uniform mesh. In the description 
  that follows we have split every step into two sub-steps. The first sub-step is the time evolution, which is due to the numerical scheme and is governed by the 
  evolution Req.(\ref{EvolutionReq}) and the second sub-step is the spatial modification, which is due to the mesh relocation and the solution update procedure and is 
  governed by the $\lambda$-rule Req.(\ref{l-ruleReq}) and the coupling Req.(\ref{CouplingReq}).

  \begin{figure}[t]
    \begin{center} \input{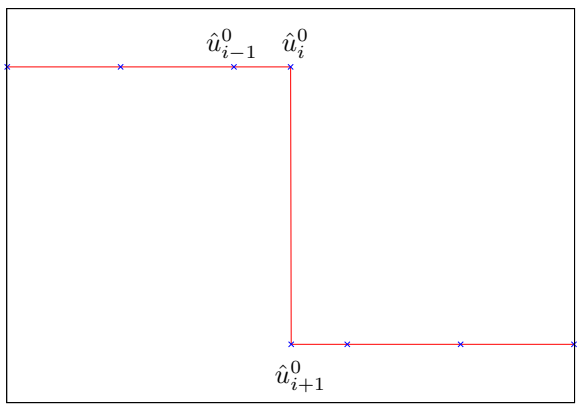} \end{center}
    \caption{This is the initial condition. In this configuration we set $\hat a_1=|\hat u_i^0-\hat u_{i+1}^0|$}\label{Graph.Initial.Condition}
  \end{figure}
  \begin{itemize}
    \item[\underline{1-st step}]
      We refer to Fig.(\ref{Graph.Initial.Condition}) for a graphical description of the configuration. The first nodal point located at the top of the shock, that is 
      $\hat u_i^0$, will evolve according to the Evolution Req.(\ref{EvolutionReq}), which reads
      $$|u_i^1-\hat u_i^0|\leq C\max\left\{|\hat u_i^0-\hat u_{i-1}^0|,|\hat u_i^0-\hat u_{i+1}^0|\right\}$$

      Since we consider jump initial conditions, it is obvious that
      $$|\hat u_i^0-\hat u_{i-1}^0|=0\quad \mbox{ and }\quad |\hat u_i^0-\hat u_{i+1}^0|\leq TV(u^0)$$
      We denote $\hat a_1=|\hat u_i^0-\hat u_{i+1}^0|$ and define $a_1=C \hat a_1$, so the evolution Req.(\ref{EvolutionReq}) for the value $\hat u_i^0$ reads,
      $$|u_i^1-\hat u_i^0|\leq C\max\{|\hat u_i^0-\hat u_{i-1}^0|,|\hat u_i^0-\hat u_{i+1}^0|\}\leq C|\hat u_i^0-\hat u_{i+1}^0|=C\hat a_1=a_1$$
      To introduce the notation for the continuation of this work we define $E_1^{1/2}$ to be the maximum magnitude of this extreme, hence 
      $$E_1^{1/2}=|u_i^1-\hat u_i^0|=a_1$$
      To explain the symbolism, we use the letter $E$ in $E_1^{1/2}$ because we refer to the magnitude of extremes, the subscript $_1$ states that we refer to 1-st 
      extreme and the superscript $^{1/2}$ states that we have moved from the time step $k=0$, with the use of the numerical scheme (time evolution step) but the mesh 
      reconstruction step has not taken place yet. 

      We refer to Fig.(\ref{Graph.l-rule1Ext}) for a graphical description. We then perform the mesh reconstruction step and because of the $\lambda$-rule 
      Req.(\ref{l-ruleReq}) and of the $\lambda$-rule effect Rem.(\ref{Rem.l-rule}) the new extreme will be of magnitude bounded by 
      $$E_1^1=\lambda a_1$$
      where full superscript $^1$ is used since the relocation has taken place. 
      \begin{figure}[t]\begin{center}
        \input{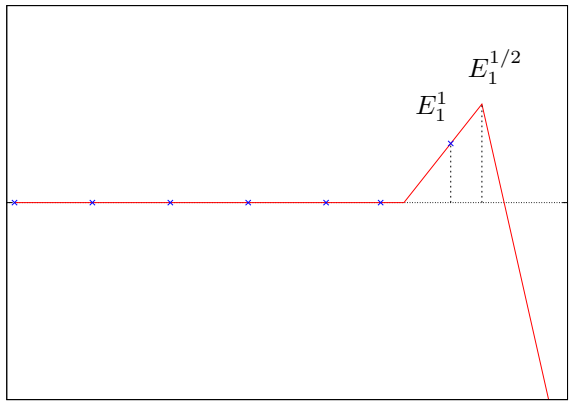}
        \caption{The resulting situation at the head of the shock at the end of the 1-st time step. Only one extreme exists in this time step and it is of 
          magnitude $E_1^1$. The numerical solution -before the remeshing procedure takes place- is depicted red, the new nodes -that occur after the remeshing 
          procedure- are depicted in blue. The new node -that is closest to the old extreme- avoids the extreme by the $\lambda$-rule resulting in a magnitude bounded 
          by $E_1^1\leq \lambda a_1$. The reconstruction of the numerical solution over the new mesh results in the clipping of the magnitude of the extreme according 
          to the $\lambda$-rule.}\label{Graph.l-rule1Ext}
      \end{center}\end{figure}
    \begin{remark}
      Regarding the continuation; the 1-st extreme shall "pollute" its neighbour by provoking the appearance of a 2-nd extreme of the opposite direction.
    \end{remark}
\end{itemize}
\begin{remark}($E_m^k$ Notation)
  We denote by $E_m^{k+1/2}$ (half superscript) the bound on the magnitude of the $m$-th extreme at the $k$-th time step after time evolution and before the mesh 
  reconstruction procedure. We denote by $E_m^{k+1}$ (full superscript) the \underline{bound} on the magnitude of $m$-th extreme at the end of the $k$-th time step, 
  that is after the time evolution and the mesh reconstruction procedure.
\end{remark}

\begin{itemize}
  \item[\underline{2-nd step}]
    We refer Figure \ref{Graph.l-rule1Ext} for the situation at the end of the 1-st step. At the end of the previous step, we had one extreme of magnitude 
    bounded by $E_1^1=\lambda a_1$. Due to the time evolution we expect the 1-st extreme to evolve to a new value. We also expect the creation 
    of a 2-nd extreme at the left side of the 1-st extreme. We will study each extreme separately.
    \begin{itemize}
      \item[1-st Extr.]
        The Evolution Req.(\ref{EvolutionReq}) dictates that this extreme shall evolve as:
        $$|u_i^{1+1/2}-\hat u_i^1|\leq C\max\left\{|\hat u_i^1-\hat u_{i-1}^1|,|\hat u_i^1-\hat u_{i+1}^1|\right\}$$
        where we use half superscript in $u$ since the relocation procedure has not taken place yet. From the previous time step we have that, 
        $$|\hat u_i^1-\hat u_{i-1}^1|\leq E_1^1\quad\mbox{ and }\quad |\hat u_i^1-\hat u_{i+1}^1|\leq 2E_1^1+\hat a_2$$
        To justify the second inequality we return at the end of the time step $k=1$ and notice that the node $i+1$ is placed along the shock, which is -by symmetry- 
        of variation at most $E_1^1+TV(u^0)+E_1^1$. So the Evolution Req.(\ref{EvolutionReq}) for the 1-st Extreme reads,
        $$|u_i^{1+1/2}-\hat u_i^1|\leq C(2E_1^1+\hat a_2)=2C E_1^1+a_2$$
        where we have defined $a_2=C\hat a_2$. If now we set $\upsilon$ to be the level from which we measure the magnitudes of the extremes (in this case the top of 
        the shock), the previous bound recasts,
        $$|(u_i^{1+1/2}-\upsilon)-(u_i^1-\upsilon)|\leq 2C E_1^1+a_2$$
        By setting $E_1^{1+1/2}=u_i^{1+1/2}-\upsilon$ and since $E_1^1=u_i^1-\upsilon$ we deduce that the magnitude of the 1-st extreme will be bounded as
        $$E_1^{1+1/2}\leq E_1^1+2C E_1^1+a_2$$
        Now the relocation procedure takes place and because of the $\lambda$-rule Req.(\ref{l-ruleReq}) and of the $\lambda$-rule effect 
        Rem.(\ref{Rem.l-rule}), the magnitude of the 1-st extreme at the end of this step shall be bounded as follows,
        $$E_1^2=\lambda (E_1^1+2CE_1^1+a_2)$$
      \item[2-nd Extr.]
        The Evolution Req.(\ref{EvolutionReq}) dictates that this extreme shall evolve as,
        $$|u_{i-1}^{1+1/2}-\hat u_{i-1}^1|\leq C\max\left\{|\hat u_{i-1}^1-\hat u_{i-2}^1|,|\hat u_{i-1}^1-\hat u_{i}^1|\right\}$$
        where again half superscript is used on $u_{i-1}^{1+1/2}$ since the relocation procedure has not taken place yet. From the previous time step we know that 
        $$|\hat u_{i-1}^1-\hat u_{i-2}^1|=0\quad\mbox{ and }\quad |\hat u_{i-1}^1-\hat u_{i}^1|\leq E_1^1$$
        So the Evolution Req.(\ref{EvolutionReq}) recasts, for the 2-nd extreme as follows, 
        $$|u_{i-1}^{1+1/2}-\hat u_{i-1}^1|\leq CE_1^1$$
        or by noting that $u_{i-1}^1=\upsilon$ is the level from which we measure the magnitudes of the extremes, the previous 
        bound recasts
        $$E_2^{1+1/2}\leq CE_1^1$$
        where, as we explained earlier half superscript is used because the relocation procedure has not taken place yet.

        Now the relocation procedure takes place and the $\lambda$-rule Req.(\ref{l-ruleReq}) dictates that the magnitude of the 2-nd extreme at the end of this step 
        shall be bounded as follows,     
        $$E_2^2=\lambda CE_1^1$$
    \end{itemize}
  So at the end of the 2-nd step the bounds on the existing extremes are as follows,
  $$E_1^2=\lambda (E_1^1+2CE_1^1+a_2),\quad E_2^2=\lambda C E_1^1$$
  The situation at the end of this step is depicted in Figure \ref{Graph.l-rule2Ext}.
    \begin{figure}[t]\begin{center}
      \input{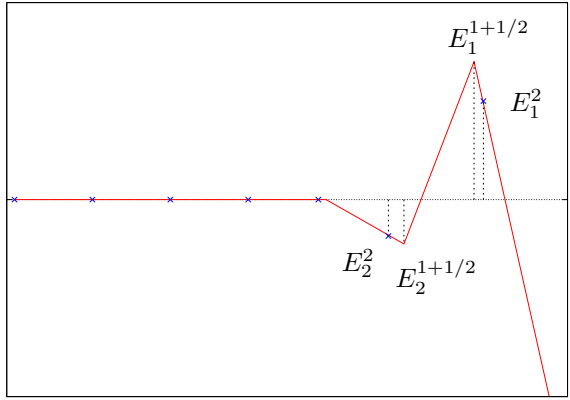}
      \caption{The resulting situation at the end of the 2-nd time step. Two extremes of magnitudes $E_1^2$ and $E_2^2$ exist in this time step. The numerical 
               solution before the remeshing procedure is depicted in red and the new nodes -after the remeshing procedure are depicted in blue. The new nodes -that 
               are closest to the previous extremes- avoid the extremes by the $\lambda$-rule}\label{Graph.l-rule2Ext}
    \end{center}\end{figure}
    \begin{remark}
      The 2-nd extreme shall provoke the appearance of a new extreme of the opposite direction. This is the pollution process.
    \end{remark}

  \item[\underline{3-rd step}]
    At the end of the previous step we had two extremes with magnitudes bounded by $E_1^2$ and $E_2^2$, see Figure \ref{Graph.l-rule2Ext}. In this step we expect them 
    to evolve to new values $E_1^3$ and $E_2^3$, we also expect a new extreme to appear, namely $E_3^3$.
    \begin{itemize}
      \item[1-st Extr.]
        Following the discussion of the previous steps, we note that the evolution of the 1-st extreme will be governed by the Evolution Req.(\ref{EvolutionReq}), so
        $$|u_i^{2+1/2}-\hat u_i^2|\leq C\max\left\{|\hat u_i^2-\hat u_{i-1}^2|,|\hat u_i^2-\hat u_{i+1}^2|\right\}$$
        where from the previous time steps we note
        $$|\hat u_i^2-\hat u_{i-1}^2|\leq E_1^2+E_2^2\quad \mbox{ and }\quad |\hat u_i^2-\hat u_{i+1}^2|\leq 2E_1^2+\hat a_3$$
        We also note that from the previous time step the bound $E_1^2=\lambda (E_1^1+2CE_1^1+a_2)$ is obviously larger than the bound  
        $E_2^2=\lambda C E_1^1$ 
        hence the Evolution Req.(\ref{EvolutionReq}) for the 1-st extreme reads as follows (after the subtraction of $\upsilon$), 
        $$E_1^{2+1/2}\leq E_1^2+2CE_1^2+a_3$$

        Now the mesh reconstruction procedure takes place and the new magnitude of the 1-st extreme shall be bounded by 
        $$E_1^3=\lambda(E_1^2+2CE_1^2+a_3)$$
      \item[2-nd Extr.]
        Using similar arguments as before, the Evolution Req.(\ref{EvolutionReq}) dictates the evolution of the 2-nd extreme as follows, 
        $$|u_{i-1}^{2+1/2}-\hat u_{i-1}^2|\leq \max\left\{|\hat u_{i-1}^2-\hat u_{i-2}^2|,|\hat u_{i-1}^2-\hat u_i^2|\right\}$$
        From the previous time step we note 
        $$|\hat u_{i-1}^2-\hat u_{i-2}^2|\leq E_2^2\quad\mbox{ and }\quad|\hat u_{i-1}^2-\hat u_i^2|\leq E_2^2+E_1^2$$
        hence the Evolution Req.(\ref{EvolutionReq}) for the 2-nd extreme yields,
        $$E_2^{2+1/2}\leq E_2^2+C(E_2^2+E_1^2)$$

        Now, relocation takes place and the new 2-nd extreme shall be of magnitude bounded by 
        $$E_2^3= \lambda(E_2^2+C(E_2^2+E_1^2))$$
      \item[3-rd Extr.]
        Repeating the work we did for the 2-nd extreme in the previous time step, the magnitude of the 3-rd extreme after both the time evolution and the node 
        relocation procedure will be bounded 
        $$E_3^3=\lambda CE_2^2$$ 
    \end{itemize}
   So at the end of the 3-rd step the bounds on the magnitudes of the existing extremes are as follows, 
    \begin{align}
      E_1^3&=\lambda(E_1^2+2CE_1^2+a_3)= \lambda^3 (1+2C)^2 a_1+\lambda^2 (1+2C) a_2+\lambda a_3,\nonumber\\
      E_2^3&=\lambda(E_2^2+C(E_2^2+E_1^2))= \lambda^3 2C(1+2C) a_1+\lambda^2 C a_2,\nonumber\\
      E_3^3&=\lambda CE_2^2=\lambda ^3 C^2 a_1\nonumber
    \end{align}
    Figure \ref{l-rule3Ext graph} depicts the situation at the end of the 3-rd step.
  \begin{figure}[t] \begin{center}
     \input{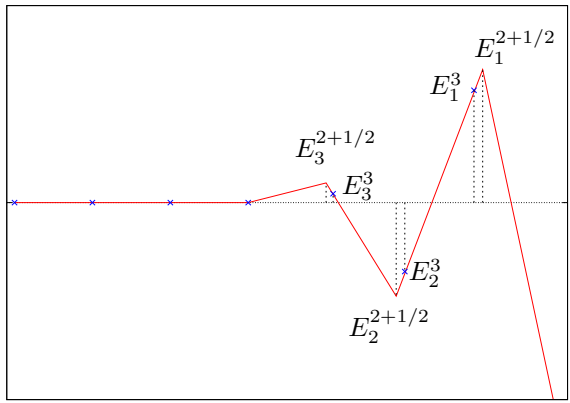} 
     \caption{The resulting situation at the end of the 3-rd time step. Three extremes of magnitude $E_1^3$, $E_2^3$, $E_3^3$ exist in this time step.}
     \label{l-rule3Ext graph}
  \end{center}\end{figure}
\end{itemize}


For the sake of completeness we define the increases $a_i$ that we used throughout the previous paragraph. For this we first analyse the variation of the 
shock at the $k$-th time step. It consists of three parts, the oscillatory part at the top of the shock with magnitude at most $E_1^k$, the main part of the shock which is of variation $TV(u^0)$ and the oscillatory part at the foot of the shock  being of magnitude at most $E_1^k$.

\begin{definition}[Definition of the $a_i$ increases]\label{a_i}
  Let $\hat u_i^k$ be the value at the top of the shock. The node $x_{i+1}^k$ is located along the shock, in one of the three parts that consist the shock.

  We define 
  $$\hat a_k=\left(|\hat u_i^k-\hat u_{i+1}^k|-2E_1^k\right)_+$$
  where the subscript $_+$ denotes the positive part. Moreover we define $a_k=C\hat a_k$.
\end{definition}
\begin{remark}
  By definition, $\hat a_k$ describes the possibly more that $2E_1^k$ distance $|\hat u_i^k-\hat u_{i+1}^k|$. That is, if $|\hat u_i^k-\hat u_{i+1}^k|<2E_1^k$ then 
  $\hat a_k=0$; hence $a_k=0$.
\end{remark}

We can generalise the situation in the $k$-th time step as follows,
\begin{itemize}
  \item[\underline{$k$-th step}]
    To generalise the description that we presented in the previous time steps, we introduce here the recursive relations for the general extreme $m$ at the general 
    time step $k\geq 1$, 
    \begin{equation}\label{ExtremeRecursion}
      \begin{cases}
        E_{m}^k= \lambda\left(E_m^{k-1}+C\cdot( E_m^{k-1}+E_{m-1}^{k-1})\right),&\mbox{for } m>1\cr
        E_1^{k}= \lambda \left(E_1^{k-1}+2C\cdot E_1^{k-1}+a_k\right),&\mbox{for } m=1\cr
      \end{cases}
    \end{equation}
  
  \begin{remark}
  From the second equation if Rel.(\ref{ExtremeRecursion}), it is evident that we add at least $2CE_1^{k-1}$ in the increase of the 1-st extreme even if the actual 
  increase needed for the highest node is less. This increases the magnitude of the 1-st extreme but at the same time simplifies the presentation and the route of the 
  proof. More precise increases, result in sharper final bounds. 
\end{remark}
In analysing these recursive relations, we see that for the evolution of the extreme $m$ -from the value $E_m^{k-1}$ to the value $E_m^k$- we take into account the neighbouring extreme in the right hand side, $E_{m-1}^{k-1}$. To justify such a choice, we have to prove that the bounds $E_m^k$ on the magnitudes of the extremes constitute a decreasing sequence with respect to $m=1,\ldots$ for every step $k$. That is we need to prove that $E_1^k>E_2^k>E_3^k>\cdots$. This is accomplished by  the following lemma, 
 
\begin{lemma}
  For every step $k$ the magnitudes of the bounds given by Rel.(\ref{ExtremeRecursion}) are in a decreasing order w.r.t $m=1,\ldots$
\end{lemma}
\begin{proof}
  Let's assume that in the step $k$, $E_{m+1}^k\leq E_m^k$ for every $m=1,\cdots$. We shall prove that $E_{m+1}^{k+1}\leq E_m^{k+1}$ for every $m=1,\ldots$
  \begin{itemize}
    \item[\underline{$m=1$}]
      The recursive relations Rel.(\ref{ExtremeRecursion}) read for $E_1^{k+1}$ and $E_2^{k+1}$ as follows,
      \begin{align}
        E_2^{k+1}&= \lambda(E_2^k+CE_2^k+CE_1^k),\nonumber \\
        E_1^{k+1}&= \lambda(E_1^k+CE_1^k+CE_1^k+a_{k+1})\nonumber
      \end{align}
      Utilising the induction hypotheses and that $a_{k+1}\geq 0$ the result $E_2^{k+1}<E_1^{k+1}$ is immediate.
    \item[\underline{$m>1$}]
      The recursive relations Rel.(\ref{ExtremeRecursion}) read for $E_m^{k+1}$ and $E_{m+1}^{k+1}$ as follows,
      \begin{align}
        E_{m+1}^{k+1}&= \lambda(E_{m+1}^k+CE_{m+1}^{k}+CE_{m}^k)\nonumber \\
        E_{m}^{k+1}  &= \lambda(E_{m}^k+CE_{m}^{k}+CE_{m-1}^k)\nonumber
      \end{align}
      The induction hypotheses states that $E_{m+1}^k\leq E_m^k\leq E_{m-1}^k$, so immediately we conclude that $E_{m+1}^{k+1}\leq E_m^{k+1}$.
  \end{itemize}
  It is so proven that in order to bound the new magnitude of every extreme we could use the recursive relations Rel.(\ref{ExtremeRecursion}).
\end{proof}
\end{itemize}

Having devised recursive relations for the evolution of the extremes, that is Rel.(\ref{ExtremeRecursion}) we continue with the study of the bounds of their magnitudes. 
\subsection{Extremes}
In this paragraph we solve the recursive relation Rel.(\ref{ExtremeRecursion}), for every extreme $m$ and we provide uniform -with respect to the time step $k$- bounds on the magnitude of the extremes.

We start by providing bounds on the extremes of $E_m^k$ with respect to the sequence of increases $a_i$. 
\begin{lemma}[Magnitude of the 1-st extreme]
  The magnitude of the first extreme in the $k$-th time step is bounded by,
  $$E_1^k\leq\lambda \sum_{j=1}^{k}\lambda^{k-j}(1+2C)^{k-j} a_j,$$
  or, by setting $l=k-j$
  \begin{equation}\label{Extreme1}
    E_1^k\leq\lambda \sum_{l=k-1}^{0}\lambda^{l}(1+2C)^{l} a_{k-l}
  \end{equation}
\end{lemma}
  \begin{proof}
    By induction. We note from the previous discussion that 
    $$E_1^1\leq \lambda a_1=\lambda \sum_{j=1}^{1}\lambda^{1-j}(1+2C)^{1-j} a_j$$ 
    For the induction hypothesis we assume that the magnitude of the 1-st extreme is bounded in the $k$-th time step as
    $$E_1^k\leq\lambda \sum^{k-1}_{l=0}\lambda^{l}(1+2C)^{l} a_{k-l}$$ 
    Using the evolution relation (\ref{ExtremeRecursion}) of the 1-st extreme, that is
    $$E_1^{k+1}\leq\lambda\left(E_1^k+2CE_1^k+a_{k+1}\right)$$
    we can bound its magnitude in the $k+1$ time step, 
    \begin{align}
      E_1^{k+1}&\leq \lambda\left((1+2C)E_1^k+a_{k+1}\right)\nonumber
    \end{align}
    The right hand side recast -by the induction hypothesis- as follows,
    \begin{align}
      E_1^{k+1}&\leq \lambda\left((1+2C)\lambda \sum_{j=1}^k \lambda^{k-j}(1+2C)^{k-j}a_{j}+a_{k+1}\right)\nonumber\\
               &\leq \lambda\left(\sum_{j=1}^k \lambda^{k+1-j}(1+2C)^{k+1-j}a_{j}+\lambda^{k+1-(k+1)}(1+2C)^{k+1-(k+1)}a_{k+1}\right)\nonumber\\
               &\leq \lambda\sum_{j=1}^{k+1} \lambda^{k+1-j}(1+2C)^{k+1-j}a_{j}\nonumber
    \end{align}
    This completes the proof regarding the bound of the magnitude of the 1-st extreme.
  \end{proof}

  We now need a similar bound on the magnitude of the 2-nd extreme,
  \begin{lemma}[Magnitude of the 2-nd extreme]
    The magnitude of the second extreme in the $k$-th time step is bounded by,
    $$E_2^k\leq \lambda^2 C\sum_{j=1}^{k-1}\binom{k-j}{k-j-1}\lambda^{k-j-1}(1+2C)^{k-j-1}a_j,$$
    or by setting $l=k-j$
    \begin{equation}\label{Extreme2}
      E_2^k\leq \lambda^2 C\sum^{k-1}_{l=1}\binom{l}{l-1}\lambda^{l-1}(1+2C)^{l-1}a_{k-l}
    \end{equation}
  \end{lemma}
  \begin{proof}
    Proof by induction using relations (\ref{ExtremeRecursion}), (\ref{Extreme1}), (\ref{Extreme2}) and the fact $\binom{n}{k}+\binom{n}{k+1}=\binom{n+1}{k+1}$.  We 
    note from the previous discussion that 
    $$E_2^2\leq \lambda^2 C a_1=\lambda^2 C\sum_{j=1}^{2-1}\binom{2-j}{2-j-1}\lambda^{2-j-1}(1+2C)^{2-j-1}a_j$$ 
    For the induction hypothesis we assume that 
    $$E_2^k\leq \lambda^2 C\sum_{j=1}^{k-1}\binom{k-j}{k-j-1}\lambda^{k-j-1}(1+2C)^{k-j-1}a_j$$
    and for the induction step we have the following
    \begin{align}
      E_2^{k+1}
        &= \lambda\left(E_2^k+C(E_2^k+E_1^k)\right)=\lambda\left((1+C)E_2^k+CE_1^k\right)\nonumber \\
        & \leq \lambda\left(\lambda^2C(1+C)\sum_{l=1}^{k-1}\binom{l}{l-1}\lambda^{l-1}(1+2C)^{l-1}a_{k-l}
          +\lambda C\sum^{k-1}_{l=0}\lambda^l(1+2C)^la_{k-l}\right)\nonumber
    \end{align}
    Where in the last step we utilised the induction hypothesis. Now, since $1+C\leq 1+2C$ the bound recasts
    \begin{align}
      E_2^{k+1}
        &\leq \lambda\left(\lambda C\sum^{k-1}_{l=1}\binom{l}{l-1}\lambda^l(1+2C)^la_{k-l}+\lambda C\sum_{l=0}^{k-1}\lambda^l(1+2C)^la_{k-l}\right)\nonumber\\
        &=    \lambda^2C\left(\sum^{k-1}_{l=1}\left(\binom{l}{l-1}+1\right)\lambda^l(1+2C)^la_{k-l}+a_k\right)\nonumber\\
        &=    \lambda^2C\left(\sum^{k-1}_{l=1}\left(\binom{l}{l-1}+\binom{l}{l}\right)\lambda^l(1+2C)^la_{k-l}+a_k\right)\nonumber\\
        &=    \lambda^2C\left(\sum^{k-1}_{l=1}\binom{l+1}{l}\lambda^l(1+2C)^la_{k-l}+a_k\right)\nonumber\\
        &=    \lambda^2C\left(\sum^{k-1}_{l=1}\binom{l+1}{l}\lambda^l(1+2C)^la_{k-l}+\binom{0+1}{0}\lambda^0(1+2C)^0a_{k-0}\right)\nonumber\\
        &=    \lambda^2C\sum^{k-1}_{l=0}\binom{l+1}{l}\lambda^l(1+2C)^la_{k-l}\nonumber
    \end{align}
    Finally we set $\mu=l+1$ and the bound on the magnitude of the 2-nd extreme reads,
    \begin{align}
      E_2^{k+1}
        &\leq    \lambda^2C\sum^{(k+1)-1}_{\mu=1}\binom{\mu}{\mu-1}\lambda^{\mu-1}(1+2C)^{\mu-1}a_{k+1-\mu}
    \end{align}
    and this completes the proof regarding the magnitude of the 2-nd extreme.
  \end{proof}

  Similarly we prove that the magnitude of the 3-rd extreme $k$-th time step is bounded by,
  $$E_3^k\leq \lambda^3 C^2\sum_{j=1}^{k-1}\binom{k-j}{k-j-2}\lambda^{k-j-2}(1+2C)^{k-j-2}a_j,$$
  or by setting $l=k-j$,
  $$E_3^k\leq \lambda^3 C^2\sum^{k-1}_{l=2}\binom{l}{l-2}\lambda^{l-2}(1+2C)^{l-2}a_{k-l}$$

  We can generalise the previous lemmas, in a compact form for the $m$-th extreme in the $k$-th time step. 
  \begin{lemma}[Magnitude of the $m$-th extreme]
    The magnitude of the $m$-th extreme in the $k$-th time step is bounded by,
    $$E_m^k\leq \lambda^m C^{m-1}\sum_{j=1}^{k-1}\binom{k-j}{k-j-m+1}\lambda^{k-j-m+1}(1+2C)^{k-j-m+1}a_{j},$$
    or by setting $l=k-j$,
    \begin{equation}\label{Extremem}
      E_m^k\leq \lambda^m C^{m-1}\sum^{k-1}_{l=m-1}\binom{l}{l-m+1}\lambda^{l-m+1}(1+2C)^{l-m+1}a_{k-l}
    \end{equation}
  \end{lemma}
  \begin{proof} The proof is exactly the same as in the 2-nd extreme and is omitted.\end{proof}

  The last lemmas provided bounds on the magnitudes of the extremes. In the following remark we merge these bounds in a single relation valid for every 
  $m=1,2,3,\ldots$.
  \begin{remark}
    The bound we have extracted for the $m$-th extreme at the $k$-th time step, that is Rel.(\ref{Extremem}):
    $$E_m^k\leq \lambda^m C^{m-1}\sum^{k-1}_{l=m-1}\binom{l}{l-m+1}\lambda^{l-m+1}(1+2C)^{l-m+1}a_{k-l}$$
    is valid for every $m=1,2,3,\ldots$, -not just for $m>1$- since, for $m=1$ the bound we extracted for the 1-st extreme at the $k$-th time step, that is 
    Rel.(\ref{Extreme1}),
    $$E_1^k\leq\lambda \sum^{k-1}_{l=0}\lambda^{l}(1+2C)^{l} a_{k-l}$$
    can be written in the form
    $$E_1^k\leq \lambda^1 C^{1-1}\sum^{k-1}_{l=1-1}\binom{l}{l-1+1}\lambda^{l-1+1}(1+2C)^{l-1+1}a_{k-l}$$
  \end{remark}
    So far we have described the creation and evolution of the extremes. We provided Recursive relations (\ref{ExtremeRecursion}) that connect the magnitudes of 
    extremes, we have solved the recursions and merged the magnitudes of the extremes into a single relation (\ref{Extremem}). 

\begin{remark}
  The bound we have extracted for the $m$-th extreme at the $k$-th time step, that is Rel.(\ref{Extremem}):
  $$E_m^k\leq \lambda^m C^{m-1}\sum_{l=k-1}^{m-1}\binom{l}{l-m+1}\lambda^{l-m+1}(1+2C)^{l-m+1}a_{k-l}$$
  is valid for every $m=1,\ldots$ since the bound we extracted for the 1-st extreme at the $k$-th time step, that is Rel.(\ref{Extreme1}):
  $$E_1^k\leq\lambda \sum_{l=k-1}^{0}\lambda^{l}(1+2C)^{l} a_{k-l}$$
  since the former can be written in the form of relation (\ref{Extremem}) for $m=1$.
\end{remark}

    Now, we note that the bounds on the magnitudes of the extremes, that we have extracted, depend on the time step $k$. We shall bound the magnitudes of the 
    extremes uniformly with respect to the time step $k$. This will allow us to provide the final proof regarding the total variation increase due to oscillations. 
   \begin{lemma}[Uniform -with respect to the time step $k$- bound on the extremes]
      If there is a constant $M>0$ such that $a_i\leq C M$ for every $i=0,\cdots\infty$ and if $\lambda+2\lambda C<1$ then every extreme $m$ 
      is uniformly -with respect to the time step $k$- bounded as,
      \begin{equation}\label{ExtremeUniform}
        E_m^k\leq M\left(\frac{\lambda C}{1-\lambda-2\lambda C}\right)^m
     \end{equation}
    \end{lemma}
  \begin{proof}
    The magnitude of the bound of the $m$-th extreme, $m=1,2,3,\ldots$ at the $k$-th time step, with $m\leq k$, is given by the Rel.(\ref{Extremem}),
    $$E_m^k\leq \lambda^m C^{m-1}\sum^{k-1}_{l=m-1}\binom{l}{l-m+1}\lambda^{l-m+1}(1+2C)^{l-m+1}a_{k-l}$$
    Since the increase $a_i$ are uniformly bounded, $a_i\leq CM$ (we refer to the definition of the increases $a_i$ Def.(\ref{a_i})) we can bound the extremes as,
    $$E_m^k\leq \lambda^m C^{m}M\sum^{k-1}_{l=m-1}\binom{l}{l-m+1}\lambda^{l-m+1}(1+2C)^{l-m+1}.$$
    Setting $\nu=l-m+1$ the previous relation reads,
    $$E_m^k\leq \lambda^m C^{m}M\sum^{k-m}_{\nu=0}\binom{\nu+m-1}{\nu}\lambda^{\nu}(1+2C)^{\nu}.$$
    All the terms inside the sum are positive, hence the right hand side of the previous relation is increasing with respect to $k$. Hence it can be bounded for 
    $k=\infty$ as follows,
    $$E_m^k\leq \lambda^m C^{m}M\sum_{\nu=0}^{\infty}\binom{\nu+m-1}{\nu}\lambda^{\nu}(1+2C)^{\nu}$$
    or
    $$E_m^k\leq \lambda^m C^{m}M\sum_{\nu=0}^{\infty}\binom{\nu+m-1}{\nu}(\lambda+2\lambda C)^{\nu}$$
    For the convergence of the previous infinite sum we recall at this point the power series expansion.
    $$\sum_{\nu=0}^\infty \binom{\nu+m-1}{\nu}t^\nu=\frac{1}{(1-t)^m},\quad\mbox{ whenever } |t|<1,$$
    and since $\lambda +2\lambda C<1$ the last bound on $E_m^k$ reads as follows,
    \begin{align}
      E_m^k&\leq \lambda^m C^{m}M\frac{1}{(1-\lambda-2\lambda C)^m}\nonumber\\
           &=M\left(\frac{\lambda C}{1-\lambda-2\lambda C}\right)^m\nonumber
    \end{align}
    Which proves the assertion of the lemma.
  \end{proof}

\begin{remark}
  If moreover we assume $\lambda+3\lambda C<1$ -instead of $\lambda+2\lambda C<1$- then the sequence of bounds on the extremes $\{E_m^k\}$ is decreasing with respect 
  to $m$, since they can be written as,
  $$E_m^k\leq M\left(\frac{\lambda C}{1-\lambda-2\lambda C}\right)^m$$
  hence
  $$\lim_{m\rightarrow \infty}E_m^k\leq \lim_{m\rightarrow \infty}M\left(\frac{\lambda C}{1-\lambda-2\lambda C}\right)^m=0$$
  since $E_m^k\geq 0$ and the fraction $\frac{\lambda C}{1-\lambda-2\lambda C}<1$ because $\lambda+3\lambda C<1$.
\end{remark}

We are ready now to measure the total variation increase due to the oscillations. 
\subsection{Variation}
  In the previous lemma we proved that each extreme separately is of bounded magnitude, uniformly with respect to the time steps $k$. The next theorem is the basic one 
  and states that in addition to the magnitude of the extremes, also the sum of the extremes is also bounded uniformly with respect to the time step $k$.
  \begin{theorem}[Main Result]
    We assume that the requirements Req.(\ref{EvolutionReq}) and Req.(\ref{l-ruleReq}) are satisfied for $\lambda$ such that $\lambda+3\lambda C<1$  by the numerical 
    scheme and the mesh. We more over assume that the sequence $\{a_i,i=1,\infty\}$ is uniformly bounded $a_i\leq CM$. Then the sum of the magnitudes of the extremes 
    is uniformly -with respect to the time step $k$- bounded as follows,
    $$\sum_{m=1}^k E_m^k\leq M \frac{1-\lambda -2\lambda C}{1-\lambda-3\lambda C}$$
  \end{theorem}
  \begin{proof}
    We shall utilise relation (\ref{ExtremeUniform}), which is valid since the requirements of the relevant lemma are satisfied. At the end of the $k$-th step we have 
    $k$ extremes $E_1^k,E_2^k,\ldots,E_k^k$. The sum  -with respect to $m$- of their magnitudes can be bounded as,
    \begin{align}
      \sum_{m=1}^k E_m^k
           &\leq M\sum_{m=1}^k\left(\frac{\lambda C}{1-\lambda -2\lambda C}\right)^m
             \leq M\sum_{m=1}^\infty\left(\frac{\lambda C}{1-\lambda -2\lambda C}\right)^m=M\frac{1}{1-\frac{\lambda C}{1-\lambda -2\lambda C}}\nonumber \\
           &\leq M \frac{1-\lambda -2\lambda C}{1-\lambda-3\lambda C}\nonumber
    \end{align}
    where the second inequality and the equality are valid since  $\lambda +3\lambda C<1$.
  \end{proof}

Summarising and concluding we can state the following theorem, which constitute our target result on the Total Variation Increase.
\begin{theorem}[Total Variation Increase Bound]
  Given the requirements of the previous Theorem, the Total Variation increase due to the oscillations is bounded and given by
  \begin{equation}\label{1stVariationBound}
    \mathrm{TVI}\leq 2M \frac{1-\lambda -2\lambda C}{1-\lambda-3\lambda C}
  \end{equation}
\end{theorem}
\begin{proof}
  The variation of the oscillatory part is bounded by twice the magnitude of the extremes. So, 
  $$\mathrm{TVI}\leq 2 M\sum_{m=1}^\infty\left(\frac{\lambda C}{1-\lambda -2\lambda C}\right)^m\leq 2M \frac{1-\lambda -2\lambda C}{1-\lambda-3\lambda C}$$
  where the last inequality results from the previous Theorem (Main Result).
\end{proof}

Although we proved our main result, we can gain better insight if we study the contribution each increase factor $a_i$ has on the total variation. For this reason we include the following paragraph.
\subsection{Variation-Revisited}\label{tv-revisit}
We shall follow now another approach that will provide us a with further insight of the "pollution" process and with a sharper bound on the Total Variation Increase. 

This approach differs from the previous one in the sense that instead of adding directly the magnitudes of the extremes $E_m^k$, we compute the contributions of the increase terms $a_i$, $i=1,2,3,\ldots$ in the each one of the extremes $E_m^k$ separately. Then we add this contributions with respect to $a_i$. 

\begin{itemize}
\item[\underline{$a_1$ cont.}]{The contribution of $a_1$ in the $k$-th step,\\
  In the $k$-th time step there exist $k$ extremes and the increase factor $a_1$ is present in each one of these extremes. So we extract the contribution of   
  $a_1$ in all the extremes that are produced during this procedure up to the $k$-th time step.

  The contribution of $a_1$ in the 1-st extreme in the $k$-th time step is given by the relation (\ref{Extreme1}) and reads as
  $$\lambda \lambda^{k-1}(1+2C)^{k-1}$$
  and in the general extreme $m$ the contribution of $a_1$ is 
  $$\lambda^mC^{m-1}\binom{k-1}{k-m}\lambda^k (1+2C)^{k-m}$$
  Summing these contributions with respect to $m$ we end up with the total contribution of $a_1$ in the $k$-th time step,
  \begin{align}
    I_{a_1}^k                  &=\sum_{m=1}^k\lambda^m C^{m-1}\binom{k-1}{k-m}\lambda^{k-m}(1+2C)^{k-m}a_1\nonumber\\
                               &=\lambda^k\sum_{m=1}^k\binom{k-1}{k-m}C^{m-1}(1+2C)^{k-m}a_1\nonumber\\
     (\mbox{for }\nu=k-m)\quad &=\lambda^k\sum^{k-1}_{\nu=0}\binom{k-1}{\nu}C^{k-1-\nu}(1+2C)^{\nu}a_1\nonumber\\
                               &=\lambda^k(1+3C)^{k-1}a_1=\lambda(\lambda+3\lambda C)^{k-1}a_1\nonumber
  \end{align}}
\item[\underline{$a_2$ cont.}]{The contribution of $a_2$ in the $k$-th step,\\
  Similarly we notice that in the $k$-th time step the increase factor $a_2$ contributes in all the extremes except the last one, $m=k$ and its total contributions is
  $$I_{a_2}^k=\lambda(\lambda+3\lambda C)^{k-2}a_2$$}
\item[\underline{$a_m$ cont.}]{The contribution of $a_m$ in the $k$-th step,\\
  More generally, in the $k$-th time step the increase factor $a_m$ contributes in all but $m-1$ extremes (the last $m-1$) in the $k$-th time step ($k\geq m$), 
  and its total contribution is
  $$I_{a_m}^k=\lambda(\lambda+3\lambda C)^{k-m}a_m$$}
\end{itemize}

\begin{remark}
  We note here that as long as $\lambda<\frac{1}{1+3C}$ i.e $\lambda+3\lambda C<1$ each one of these contributions converges to $0$ as $k\rightarrow \infty$,
  \begin{equation}
    \lim_{k\rightarrow\infty}I_{a_m}^k=\lim_{k\rightarrow\infty}\lambda(\lambda+3\lambda C)^{k-m} a_m=0\quad\mbox{if}\quad\lambda<\frac{1}{1+3C}
  \end{equation}
  This is the very essence of the $\lambda$-rule effect. Namely a remeshing procedure which respects the $\lambda$-rule requirement (\ref{l-ruleReq}) at the extremes 
  is able to limit the increase of the variation due to each $a_i$ -eventually kill it- and hence provide us with a control over the total variation of the scheme.
\end{remark}
To finalise this second approach to the Total Variation Increase due to oscillations we continue by summing the contributions of all the $a_i$'s in the $k$-th time 
step. This will result in half the Total Variation Increase due to oscillations in the $k$-th time step.

By the previous talk the following corollary is obvious,
\begin{corollary}[Total contribution in th $k$-th step]
  In the $k$-th time step we have contribution by $a_1,a_2,\ldots,a_k$, with sum,
  \begin{equation}\label{TotalContribution}
    I^k_{tot}=\sum_{m=1}^kI^k_{a_m}=\lambda\sum_{m=1}^k (\lambda+3\lambda C)^{k-m} a_m
  \end{equation}
\end{corollary}
\begin{corollary}[Result 1]
  If we assume that the sequence $a_i$, is bounded i.e there exists $M>0$ such that $a_i\leq CM$ for all $i=1,\ldots,\infty$ and that $\lambda+2\lambda C<1$ then the 
  total contribution in the $k$-th time step reads,
  \begin{equation}\label{2ndVariationBound}
    \mathrm{TVI}\leq \frac{2\lambda C}{1-(\lambda +3\lambda C)}M
  \end{equation}
\end{corollary}
\begin{proof}
  Since the increase factors $a_i$ are uniformly bounded as $a_i\leq CM$, their total contribution given in  Rel.(\ref{TotalContribution}) reads 
  \begin{align}
    I^k_{tot}&\leq \lambda CM\sum_{m=1}^k (\lambda+3\lambda C)^{k-m} \mathop{=}^{n=k-m}=\lambda CM\sum_{n=0}^{k-1}(\lambda+3\lambda C)^n\nonumber \\
             &=\lambda CM\frac{1-(\lambda+3\lambda C)^k}{1-(\lambda +3\lambda C)}\nonumber
  \end{align}
  the previous sequence, in the right hand side, is increasing with respect to the time step $k$, so by taking the limit as $k\rightarrow\infty$ we deduce an uniform 
  -with respect to $k$- bound on the total contribution
  $$I_{tot}^\infty=\lim_{k\rightarrow\infty}I_{tot}^k\leq \lambda CM\lim_{k\rightarrow\infty}\frac{1-(\lambda+3\lambda C)^k}{1-(\lambda +3\lambda C)}
                  =\frac{\lambda C}{1-(\lambda +3\lambda C)}M$$
  This is exactly the result we were looking for since now, the Total Variation Increase due to oscillations is bounded
  $$\mathrm{TVI}\leq 2\cdot I_{tot}^\infty\leq \frac{2\lambda C}{1-(\lambda +3\lambda C)}M$$
\end{proof}
\begin{corollary}[Result 2]
  If we assume that the sequence $a_i$, $i=1\ldots\infty$ is uniformly bounded as $a_i\leq CM=C\TV(u_0)$ then the previous bound on the total variation increase 
  becomes, 
  $$2\cdot I_{tot}^\infty\leq \frac{2\lambda C}{1-(\lambda +3\lambda C)}\TV(u_0)$$
\end{corollary}
\begin{remark}
  This result is even better than the previous one given Rel.(\ref{2ndVariationBound}) since the bound that provides on the increase of the Total Variation due to 
  oscillations is directly related to the variation of the initial condition $u_0$.
\end{remark}
With more delicate assumptions on the increases $a_i$ we get the following Corollary,
\begin{corollary}[Result 3]
  If we assume that the sum of all increases $\sum_{i=0}^\infty a_i$ is finite $\sum_{i=0}^\infty a_i=A<\infty$ then the Total Variation Increase due to oscillations 
  diminishes with respect to the time step $k$.
\end{corollary}
\begin{proof}
  We note that the following sums converge
  $$\sum_{i=0}^\infty (\lambda+3\lambda C)^i=\frac{1}{1-(\lambda +3\lambda C)}<\infty$$
  $$\sum_{i=0}^\infty a_i<\infty$$
  Moreover the sum 
  $$\sum_{k=1}^\infty I^k_{tot}=\sum_{k=1}^\infty \left(\lambda\sum_{j=1}^k (\lambda+3\lambda C)^{k-j} a_j\right)
                               =\lambda\sum_{k=1}^\infty \left(\sum_{j=1}^k (\lambda+3\lambda C)^{k-j} a_j\right)$$ 
  constitutes the sum of the terms of the Cauchy product of the series $\sum_{i=1}^\infty a_i$ and $\sum_{i=1}^\infty (\lambda+3\lambda C)^i$. So we deduce that the 
  following sum also converges,
  $$\sum_{k=1}^\infty I^k_{tot}<\infty$$
  and hence the sequence $I^k_{tot}$ must converge to 0, so
  $$I_{tot}^\infty=\lim_{k\rightarrow \infty}I^k_{tot}=0$$
\end{proof}
\begin{remark}
  The last result states that if the sum $\sum a_i$ is finite then the overall increase of the variation due to the oscillations produced by the $a_i$s' 
  diminishes with respect to the time step $k$.

  Moreover we note that the requirement $\sum a_i<\infty$ is supported numerically, since the first node on the shock at the right hand side of the first, is 
  always very close to the first extreme. This is due to the high density of the mesh at the shock. Hence by the definition of the increase factors Def.(\ref{a_i}), 
  $\hat a_k=0$ and so $a_k=0$.
\end{remark}

We have devised two different bounds concerning the Total Variation Increase due to oscillations. The first was immediate summation of the bounds $E_m^k$ of the magnitudes of the extremes and resulted in the bound Rel.(\ref{1stVariationBound}). The second one came by investigating the contributions of the increase factors $a_i$ and resulted in the bound Rel.(\ref{2ndVariationBound}). 

\paragraph{Comparison of the two bounds}
  We start this paragraph by restating the two bounds on the Total Variation. The first given in Rel.(\ref{1stVariationBound})
  $$B_1= 2M \frac{1-\lambda -2\lambda C}{1-\lambda-3\lambda C}$$
  and the 2-nd given in Rel.(\ref{2ndVariationBound})  
  $$B_2= \frac{2\lambda C}{1-\lambda -3\lambda C}M$$

  Before the comparison, a comment on the nature of the second bound
%

  To compare the two bounds one can examine their ratio, that is the fraction of the bound Rel.(\ref{2ndVariationBound}) over the bound of 
  Rel.(\ref{1stVariationBound}), 
  \begin{lemma}(Comparison of the bounds $B_2$ and $B_1$)
    If $\lambda+3\lambda C<1$ then $\frac{B_2}{B_1}<1$. If moreover $\lambda +4\lambda C<1$ then $\frac{B_2}{B_1}<\frac{1}{2}$
  \end{lemma}
  \begin{proof}
    The ratio of the bounds $B_2$ and $B_1$ is
    $$\frac{B_2}{B_1}=\frac{\frac{2\lambda CM}{1-\lambda-3\lambda C}}{2M\frac{1-\lambda -2\lambda C}{1-\lambda-3\lambda C}}=\frac{\lambda C}{1-\lambda-2\lambda C}$$
    If $\lambda+3\lambda C<1$ then $\frac{B_2}{B_1}<1$ since $\lambda C<1-\lambda-2\lambda C$, hence the later bound $B_2$ is sharper. If moreover 
    $\lambda +4\lambda C<1$ instead of $\lambda+3\lambda C<1$ then $\frac{B_2}{B_1}<\frac{1}{2}$ since $2\lambda C<1-\lambda-2\lambda C$.
  \end{proof}
  The meaning of this proposition is that by a careful selection of the respect factor $\lambda$ the bounds on the increase of the variation in the second approach can 
  be significantly better than that of the first approach.

\section{Main Adaptive Scheme (MAS)}\label{Section.MAS}
In this section the procedure that is responsible for the creation and manipulation of non-uniform meshes, their connection with the numerical solutions and the time evolution. We start by providing the definition of MAS once more,
\startenv{
give references
}\stopenv
\begin{definition}[Main Adaptive Scheme (MAS)]
  Given, at time step $n$, the mesh $M_x^n=\{a=x_1^n<\cdots<x_N^n=b\}$ and the approximations $U^n=\{u_1^n,\ldots,u_N^n\}$, the steps of the (MAS) are as 
  follows:
  \begin{itemize}
    \item[1.](Mesh Reconstruction)\\
      Construct new mesh $M_x^{n+1}=\{a=x_1^{n+1}<\cdots<x_N^{n+1}=b\}$
    \item[2.](Solution Update)\\
      Using the old mesh $M_x^n$ the approximations $U^n$ and the new mesh $M_x^{n+1}$:
      \begin{itemize}
         \item[2a.] construct a piecewise linear function $V^n(x)$ such that $V^n|_{M_x^n}=U^n$ 
         \item[2b.] define the updated approximations $\hat U^n=\{\hat u_i^n,\ldots,\hat u_N^n\}$ as $\hat U^n= V^n|_{M_x^{n+1}}$
      \end{itemize}
    \item[3.](Time Evolution)\\
      Use the new mesh $M_x^{n+1}$, the new approximations $\hat U^n$ and the numerical scheme to march in time and compute 
      $U^{n+1}=\{u_1^{n+1},\ldots,u_N^{n+1}\}$
    \item[4.](Loop)\\
      Repeat the Step 1.-3. with $M_x^{n+1}$, $U^{n+1}$ as initial data.
  \end{itemize}
\end{definition}

The Mesh reconstruction procedure (Step 1.) is a way of relocating the nodes of the mesh according to the geometric \emph{information} contained in  a discrete function. The basic idea of the mesh reconstruction is simple and geometric 
\begin{center}
"in areas where the numerical solution is \emph{smoother/flatter} we need less nodes where, in contrary, in areas where the numerical solution is less \emph{smooth/flat} more nodes are in order"
\end{center}

The key point in this procedure is the way that we measure the geometric \emph{information} of the discrete function. This is accomplished by using two auxiliary functions. The first one -\emph{estimator} function- measures the geometric \emph{information} of the discrete function and the second one - \emph{monitor} function- redistributes the nodes according to the \emph{information} measured by the \emph{estimator} function.

Some examples of estimator functions are the arclength estimator, the gradient estimator and the curvature estimator. Throughout this work we shall use the curvature estimator.

\paragraph{The curvature estimator function}

To start with we consider a smooth (in the classical sense) function $U$. The function $K_U(x)$ that measures the curvature of the smooth function $U$ is defined as follows,
$$K_U(x)=\frac{\left|U''(x)\right|}{\left(1+(U'(x))^2\right)^{3/2}}$$

The functions that we deal with (numerical solutions) are not smooth, they are merely point values and hence a discrete analog of the curvature estimator functions is needed. To gain a discrete analog of the estimator function we just have to discretize the derivatives that appear in the smooth estimator. Again several choices are possible, but the one that we use throughout this work is the following,
$$K_i^{dscr}=\frac{\frac{2}{x_{i+1}^n-x_{i-1}^n}\left|\frac{u_i^n-u_{i-1}^n}{x_i^n-x_{i-1}^n}-\frac{u_{i+1}^n-u_i^n}{x_{i+1}^n-x_i^n} \right|}
  {\left(\left(1+\left(\frac{u_i^n-u_{i-1}^n}{x_i^n-x_{i-1}^n}\right)^2\right)
    \left(1+\left(\frac{u_{i+1}^n-u_i^n}{x_{i+1}^n-x_i^n}\right)^2 \right)
    \left(1+\left(\frac{u_{i+1}^n-u_{i-1}^n}{x_{i+1}^n-x_{i-1}^n}\right)^2 \right)\right)^{1/2}}$$
where $M_x^n=\{x_i^n,\ i=1,\dots,N\}$ and $U^n=\{u_i^n,\ i=1,\dots,N\}$ are the non-uniform mesh, and the approximate solution at time step $n$.

With the discrete estimator we can evaluate the respective geometric property of the numerical solution itself. By performing a point by point evaluation of this discrete estimator results in a finite sequence that contains the measured information of every node $(x_i^n,u_i^n)$. That is, 
$$K_U^{dscr}=\left\{(x_1^n,K^{dscr}_1),\cdots,(x_N^n,K^{dscr}_N)\right\}$$
\begin{remark}
  We note here that the proposed mesh reconstruction procedure, assumes positive values on behalf of the discrete estimator. So now on we assume that 
  $K^{dscr}_i\geq 0$ -even though this is obvious for our discrete curvature estimator.
\end{remark}
\begin{remark}
  In areas where the numerical solution is flat, all the nodes yield $K_i^{dscr}=0$, hence no information is extracted. To avoid such a case we select a 
  $\varepsilon>0$ and set $K^{dscr}_i=\max\{K^{dscr}_i,\varepsilon\}$ for every $i$. A typical value of $\varepsilon\approx 10^{-15}$.
\end{remark}
\begin{remark}
  In the initial condition of a Riemann problem all the nodes (except for the 2 nodes at the top and bottom of the discontinuity) attain the minimum "information" 
  $K^{dscr}_i=\varepsilon$. The other two nodes attain very large "information". This abrupt change of "information" yields a non smooth non-uniform mesh. To avoid 
  such abrupt information change we use a constant $pw>0$ and set $K^{dscr}_i=(K^{dscr}_i)^{pw}$ for every $i$. This results in a smoother transition of consequent 
  discrete estimator values, a typically value of $pw\approx 0.9$.  
\end{remark}

Finally, these points, i.e $\left\{(x_1^n,K^{dscr}_1),\cdots,(x_N^n,K^{dscr}_N)\right\}$ are interpolated by a piecewise linear function $I_{K^{dscr}}$. 

We move on to the monitor function that will provide us with the new nodes. 

\paragraph{The monitor function}
The evaluation of the monitor function starts from its discrete analog, that is we first evaluate the monitor function in every old node $x_i^n$ and the we construct the continuous monitor function by linear interpolation of the discrete monitor values.  

We integrate the piecewise linear function $I_{K^{dscr}}$ to find the value of the discrete monitor function in every node $x_i^n$,
$$M_{u^n}^{x_i^n}=\int_0^{x_i^n}I_{u^n}(x)dx.$$
This results in a new sequence,
$$\left\{(x_i^n,M_{u^n}^{x_i^n}), i=0,\cdots,N)\right\}$$
This sequence is positive and strictly increasing since $K_i^{dscr}>0$ for every $i=0,\cdots, N$.

Finally, we interpolate over the values of this sequence by a piecewise linear function $M_{U^n}(x)$, which is continuous, positive and strictly increasing and so attains it maximum at the right end $M(1)$. We are ready now to relocate the nodes of the mesh. 
\begin{figure}
\begin{tabular}{cc}
  \includegraphics[width=7cm]{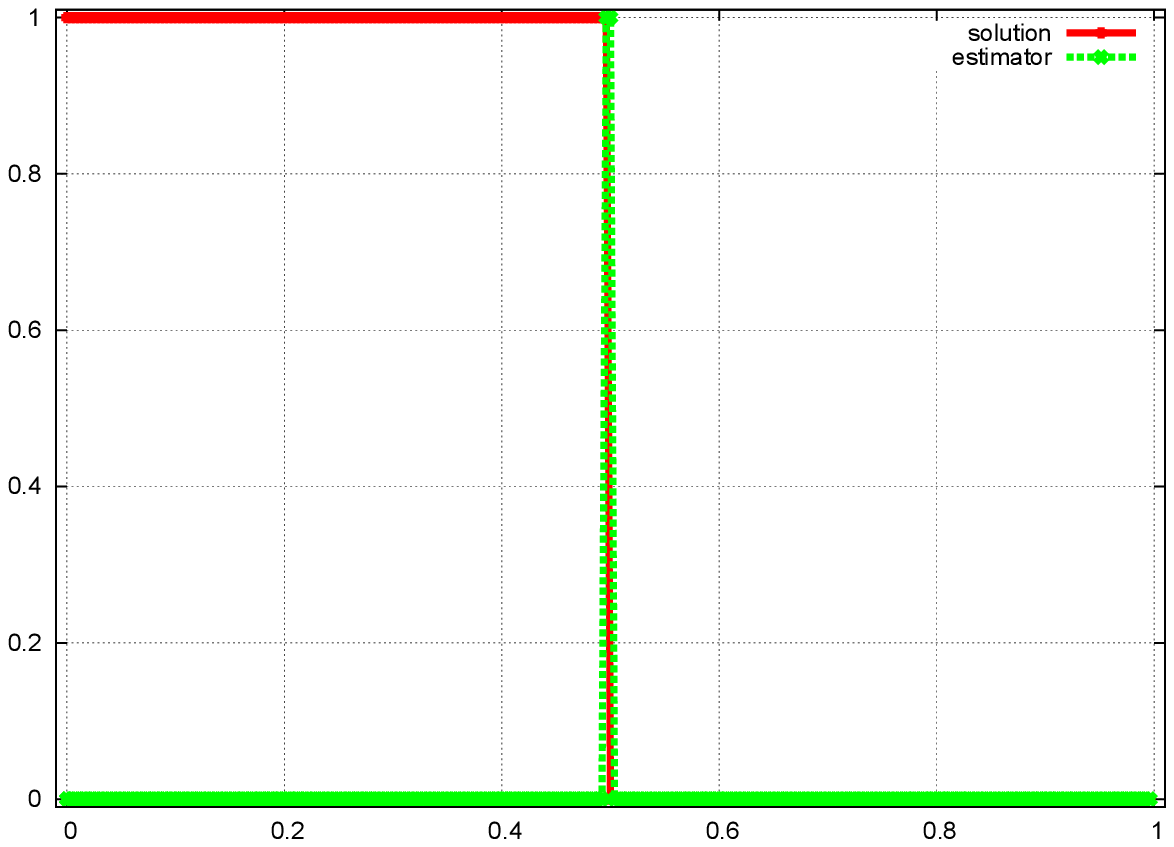}& \includegraphics[width=7cm]{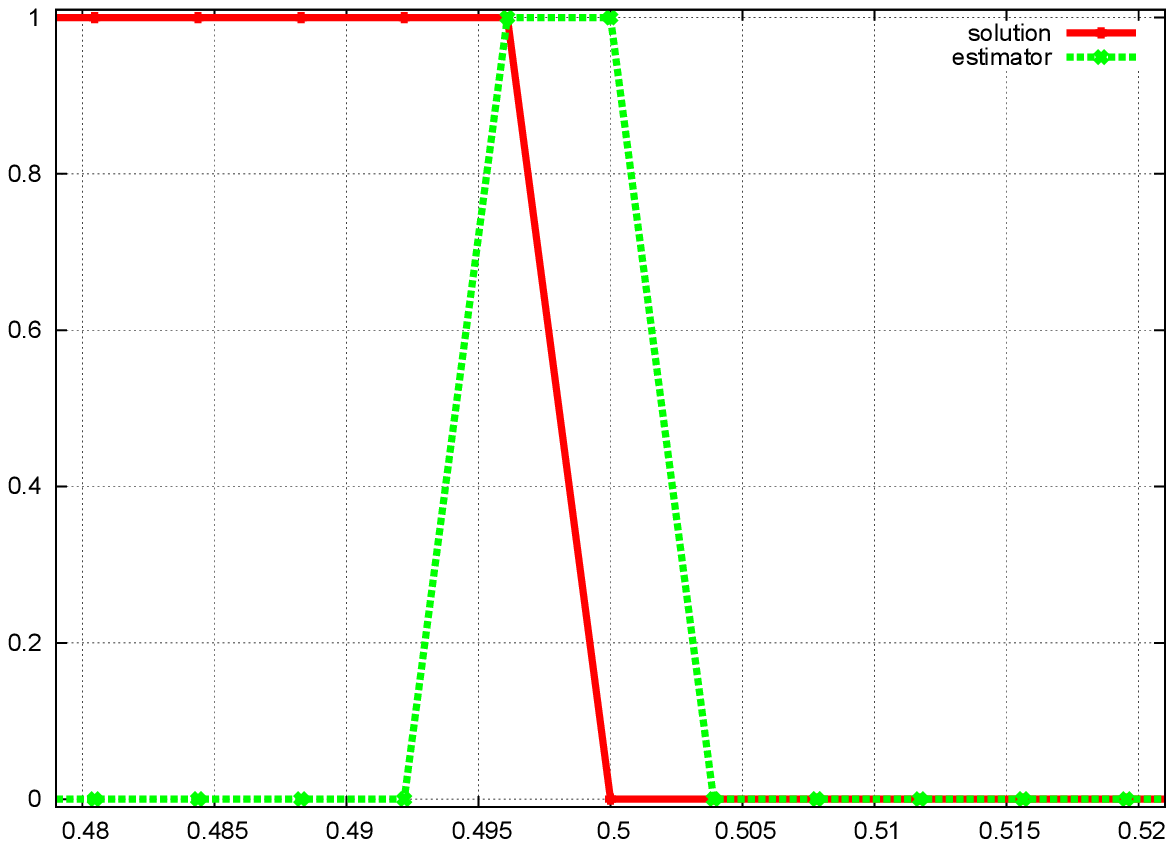}\cr
\end{tabular}
\caption{A typical estimator function is depicted along with the respective numerical solution. The right graph is a focused version of the left one.}\label{Graph.EstimatorVsSol}
\end{figure}

\begin{figure}
\begin{tabular}{cc}
  \includegraphics[width=7cm]{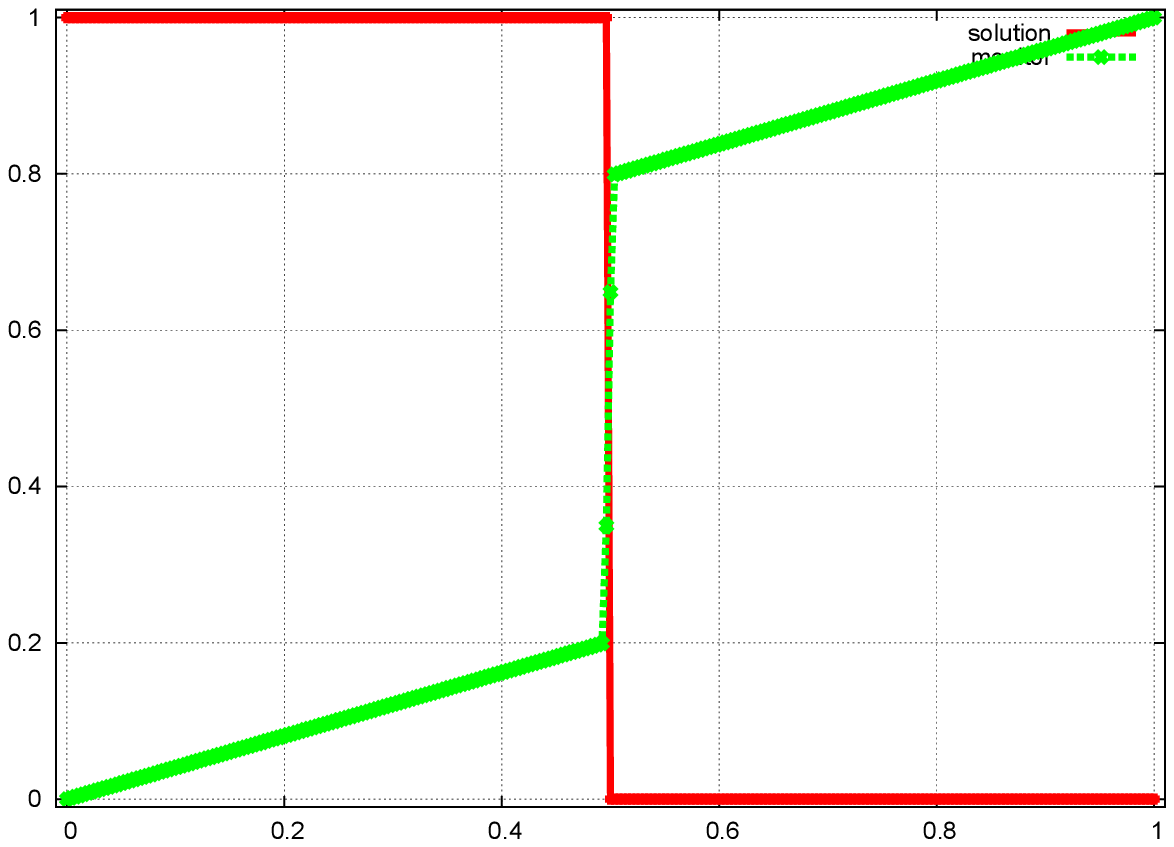}& \includegraphics[width=7cm]{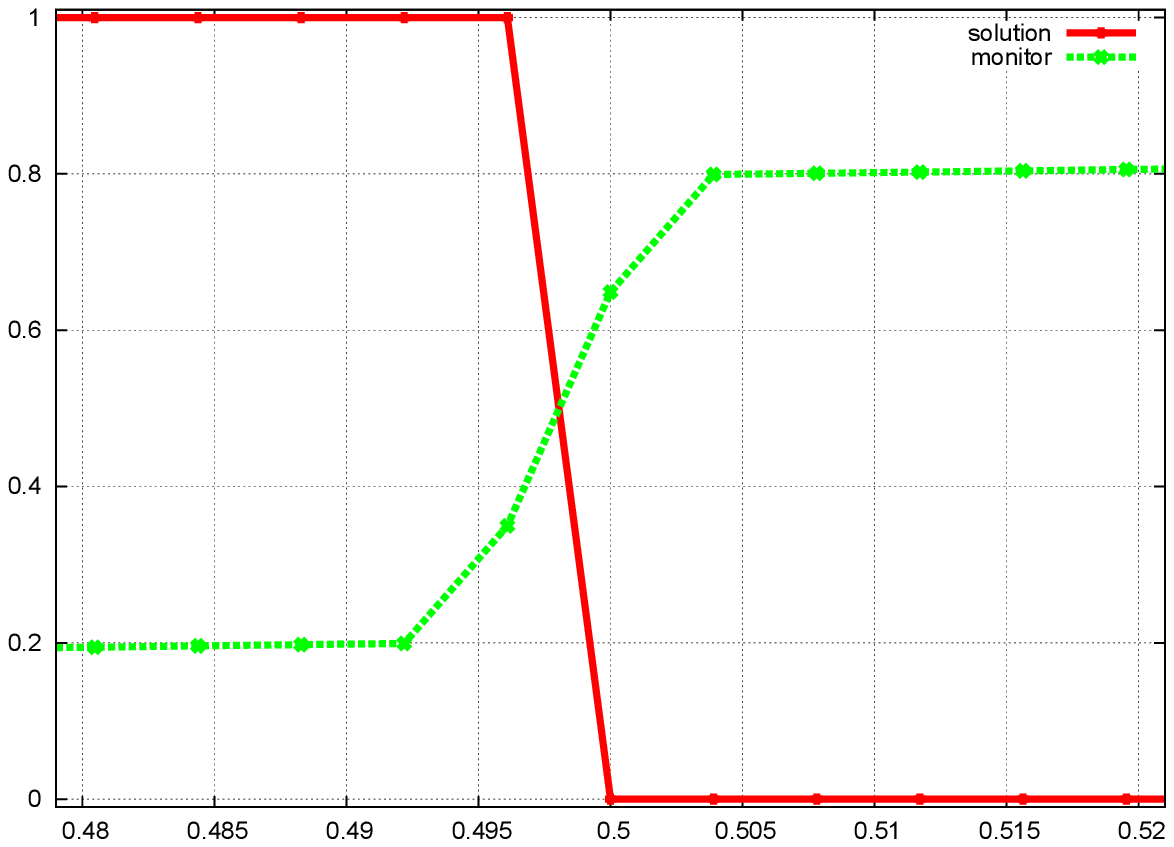}\cr
\end{tabular}
\caption{A typical monitor function is depicted along with the respective numerical solution. The right graph is a focused version of the left one.}\label{Graph.MonitorVsSol}
\end{figure}

\begin{figure}
\begin{tabular}{cc}
  \includegraphics[width=7cm]{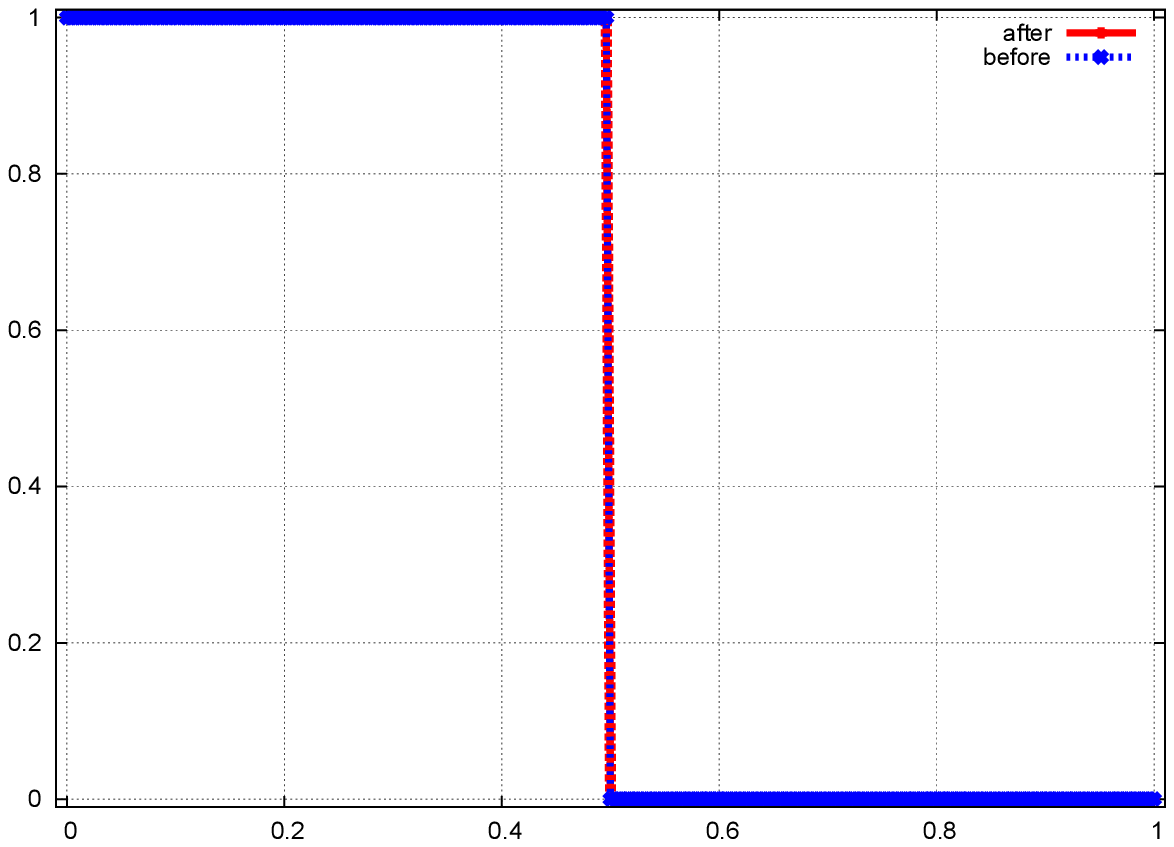}& \includegraphics[width=7cm]{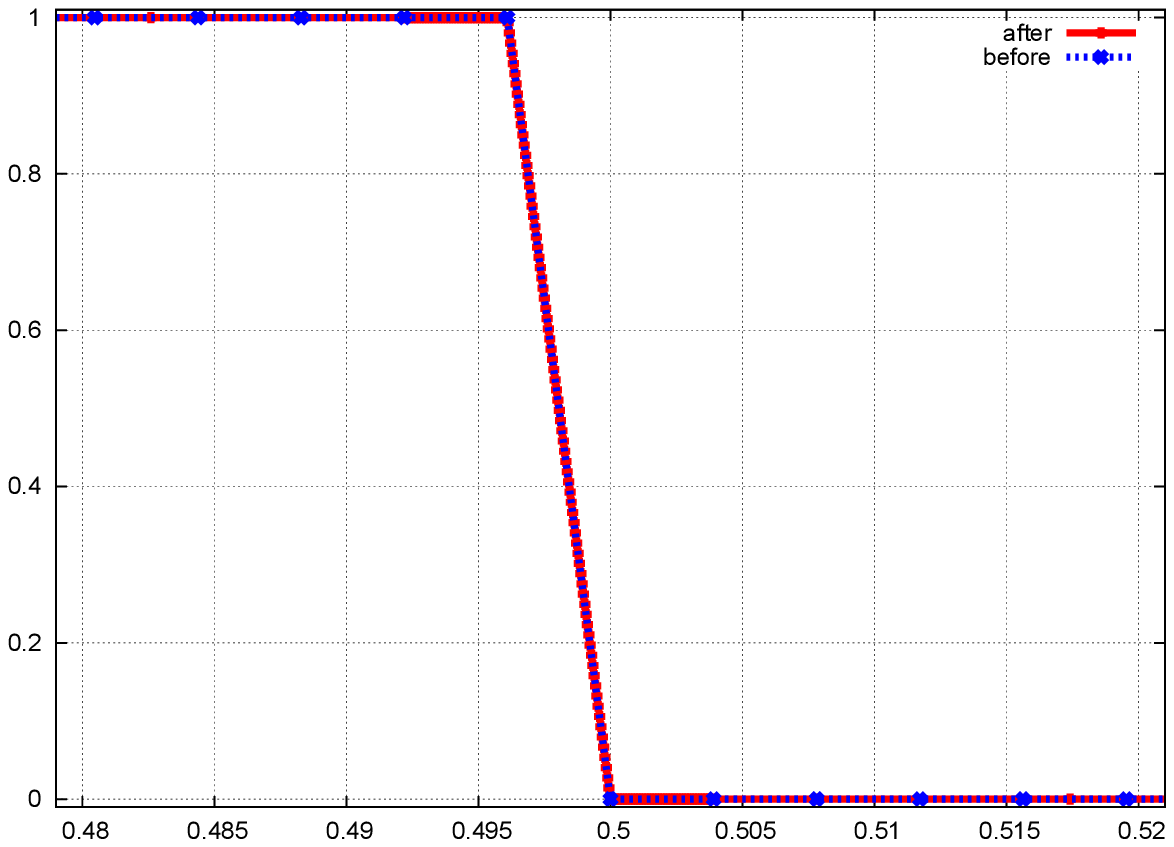}\cr
\end{tabular}
\caption{We see in this graph the result of the relocation procedure. The initial function before the mesh reconstruction procedure takes place is depicted in red. In 
  blue we can see the reconstructed numerical solution, after the mesh reconstruction procedure. We can see that the density of the nodes is higher around the area of 
  interest of the numerical solution. We also notice that there are nodes placed along the slope of the shock.}\label{Graph.BeforeVsAfter}
\end{figure}

\paragraph{Mesh reconstruction}
  What we now need is a new set of nodes, $\{x_i^{n+1},i=0,\cdots,N\}$, with $x_0^{n+1}=0$ such that they equi-distribute the total information $M(1)$ that we 
  measured. This is accomplished by solving -recursively with respect to $x_{i+1}^{n+1}$- the system,

$$\begin{cases} x_0^{n+1}=0,\cr 
                M_{U^n}(x_{i+1}^{n+1})-M_{U^n}(x_i^{n+1})=\frac{1}{N}M_{U^n}(1)
  \end{cases},\quad i=0,\cdots,N-1$$
\begin{remark}
  It is obvious the last new node $x_{N}^{n+1}$ shall be the right end of our interval, that is $x_{N}^{n+1}=1$. Moreover the fact that $M_{U^n}(x)$ is strictly 
  increasing, allows for inversion and since it is piecewise linear reduces the computational cost of the solution of the previous system to $\mathcal O(N)$.
\end{remark}

\begin{remark}
  Let us note that the number $N$ of nodes is constant. Let us also note that the mesh reconstruction procedure is not related with the numerical scheme that we use 
  for the evolution part of the problem, it is merely connected to the geometry of the numerical solution itself. 
\end{remark}

In Figures \ref{Graph.EstimatorVsSol} and \ref{Graph.MonitorVsSol} a typical Estimator and Monitor function are depicted respectively. In Figure \ref{Graph.BeforeVsAfter} one can see the affect that the node relocation procedure has on initial condition.

\section{Computational considerations}\label{Section.Considerations}
In this section we describe the numerical implementation of the Coupling Req.(\ref{CouplingReq}), and the connection of the $\lambda$-rule Req.(\ref{l-ruleReq}) with the $\lambda$-rule effect -as described in Rem.(\ref{Rem.l-rule})- when several nodes are located in between two consequent extremes.

\paragraph{Numerical implementation of the coupling requirement} 
We start by restating the coupling requirement,
\begin{requirement*}[Coupling requirement]
  The constants $C$ of the evolution Req.(\ref{EvolutionReq}) and $\lambda$ of the mesh reconstruction  Req.(\ref{l-ruleReq}) are connected via the following relation:
  \begin{equation}
    \lambda+3\lambda C<1
  \end{equation}
\end{requirement*}
In this section we shall discuss the numerical implementation of this requirement. We start by defining $I_j^n$ to be the set of indices of the nodes $x_j^{n+1}$ that are placed -after the mesh reconstruction step- "close" to a position of a local extreme of $U^n$, that is
$$I_j^n=\{j\  |\  x_j^{n+1}\in[x_i^n,x_{i+1}^n) \mbox{ for some } i \mbox{ and } U^n \mbox{ exhibits local extreme at } x_i^n \mbox{ or }x_{i+1}^n\}.$$
For every $j\in I_j^n$, $U^n$ exhibits a local extreme either at $x_i^n$ or $x_{i+1}^n$, so we set:
$$A_j=\frac{x_{i+1}^n-x_j^{n+1}}{x_{i+1}^n-x_i^n}(1+3C)$$
or respectively
$$A_j=\frac{x_j^{n+1}-x_i^n}{x_{i+1}^n-x_i^n}(1+3C)$$
where the constant $C$ is related to the numerical scheme under discussion. With this notation the coupling requirement reads for the discrete case as follows:
$$\max_{j\in I_j^n}A_j <1$$

\begin{figure}[t]
  \begin{tabular}{cc}
    \includegraphics[width=7cm]{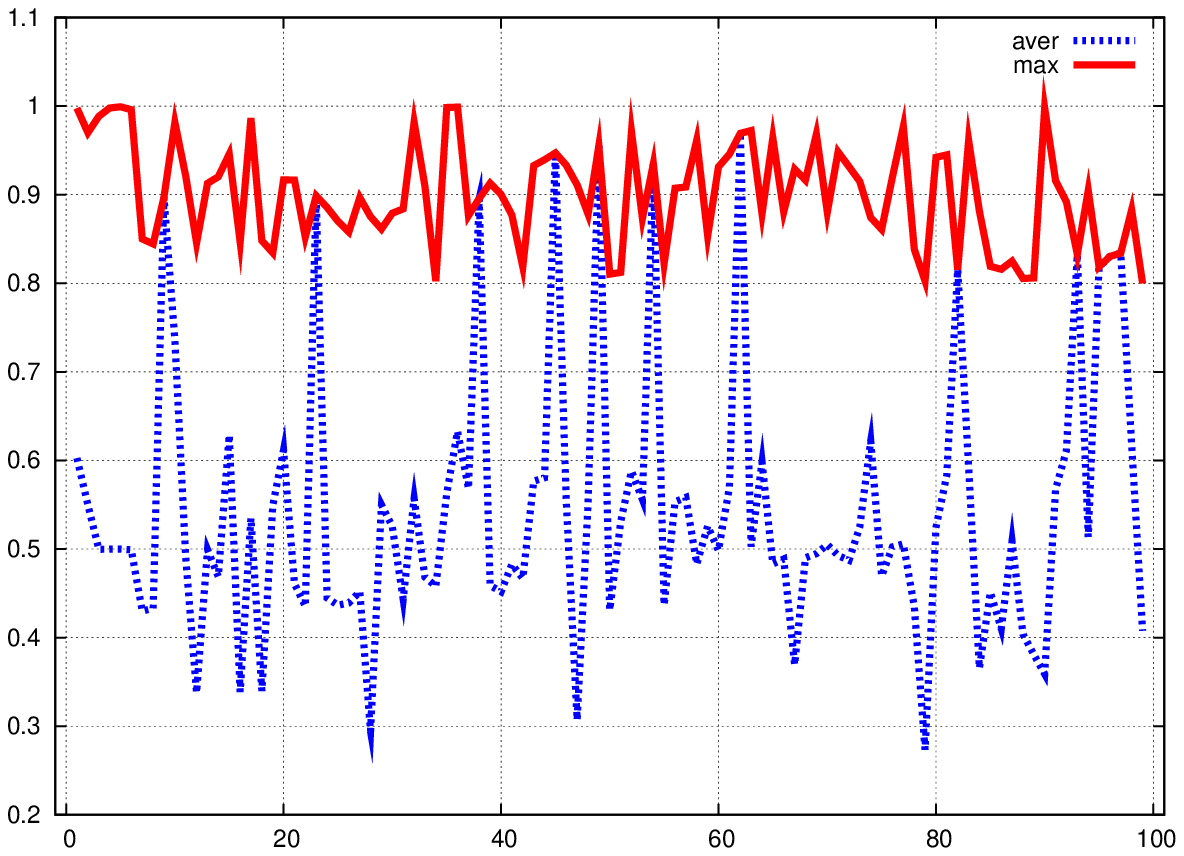}&  \includegraphics[width=7cm]{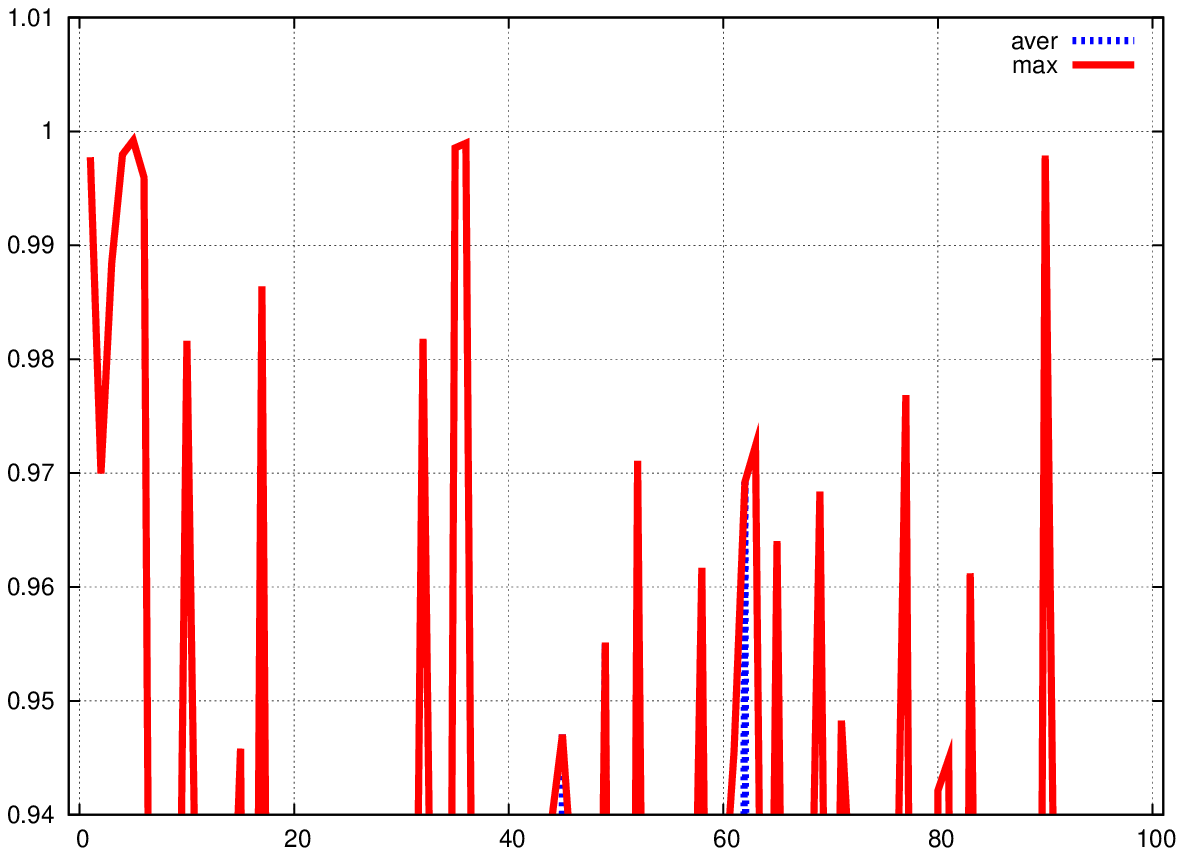}\cr
  \end{tabular}
\caption{A typical evolution of the maximum and the average $A_j$, $j\in I_j^n$ with respect to the time steps $n$ (horizontal axis). In the left graph we present the 
         full range of their evolution, while the right graph is a focused version of the left one where one can more clearly see that the value 1 is an upper bound 
         for the $\max_{j\in I_j^n}A_j$, the same is true for the average $A_j$, $j\in I_j^n$.}\label{Graph.lambda}
\end{figure}

To impose numerically this requirement we check, for every $j\in I_j^n$, whether $A_j\geq 1$, and if so, we correct accordingly the position of the node $x_j^{n+1}$. More specifically, if the local extreme is at the node $x_i^n$ we set
$$x_j^{n+1}:=x_j^{n+1}+\epsilon (x_j^{n+1}-x_i^n)$$
where $0<\epsilon<1$, with a typical value of $\epsilon=0.2$.

Similarly we treat the case where a local extreme is at the node $x_{i+1}^n$. We refer to Figure \ref{Graph.lambda} for a typical graph of the maximum and the average $A_j$, $j\in I_j^n$ with respect to time step $n$.
\paragraph{Connection of the $\lambda$-rule requirement with the $\lambda$-rule effect}
We now explain the connection of the $\lambda$-rule Req.(\ref{l-ruleReq}) with the $\lambda$-rule effect -as described in Rem.(\ref{Rem.l-rule})- when several nodes are located in between two consequent extremes.

For this, it is sufficient to show that an application of the mesh reconstruction (Step 1.) and the solution update (Step 2.) has a diffusive effect on the discrete numerical solution; since then, continuous repetitions of Step 1. and Step 2. -which do not consume any time of the physical evolution of the problem- yield the required $\lambda$-rule effect Rem.(\ref{Rem.l-rule}).

\begin{figure}[t]
  \begin{center}  \input{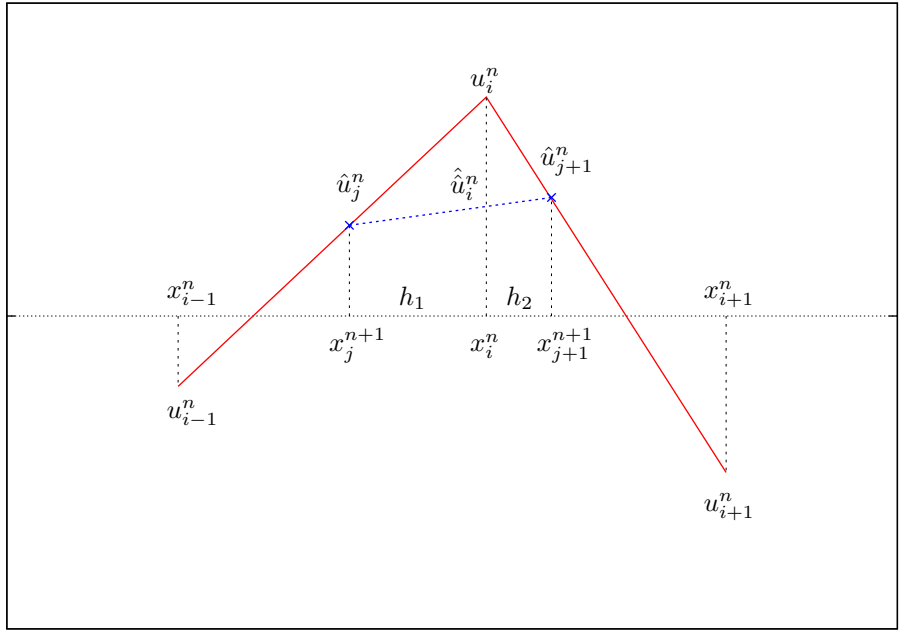}  \end{center}
    \caption{The numerical solution before the mesh reconstruction is $(x_\cdot^n,u_\cdot^n)$ and is depicted in red, after the mesh reconstruction is 
             $(x_\cdot^{n+1},\hat u_\cdot^n)$ and is depicted in blue. The solution is updated by projection of the new nodes on piecewise linears. The projected 
             value of the old node $x_i^n$ on the linear new (blue) linear segment is $\hat{\hat u}_i^n$.
            }\label{Graph.Diffusion}
\end{figure}

We refer to Figure \ref{Graph.Diffusion} and assume that a piecewise linear function attains a local maximum at the node $x_i^n$ with value $u_i^n$. The neighbouring nodes are $x_{i-1}^n$ and $x_{i+1}^n$ with values $u_{i-1}^n$, $u_{i+1}^n$ respectively. The slopes, left and right of the node $x_i^n$ are given by 
$$\lambda_i^-=\frac{u_i^n-u_{i-1}^n}{h_i^-} \quad\mbox{ and }\quad \lambda_i^+=\frac{u_{i+1}^n-u_i^n}{h_i^+},$$ 
where $h_i^-=x_i^n-x_{i-1}^n$ and $h_i^+=x_{i+1}^n-x_i^n$.

Next we assume that there is a smooth function $u$ that interpolates the points $(x_{i-1}^n,u_{i-1}^n)$, $(x_i^n,u_i^n)$ and $(x_{i+1}^n,u_{i+1}^n)$, and we compute 
\begin{align}
  \frac{\lambda_i^+-\lambda_i^-}{h_i^-+h_i^+}
       &=\frac{\frac{u_{i+1}^n-u_i^n}{h_i^+}-\frac{u_i^n-u_{i-1}^n}{h_i^-}}{h_i^-+h_i^+}
        =\frac{\frac{u(x_{i+1}^n)-u(x_i^n)}{h_i^+}-\frac{u(x_i^n)-u(x_{i-1}^n)}{h_i^-}}{h_i^-+h_i^+}\nonumber\\
       &=\frac{u'(x_i^n)+\frac{h_i^+}{2}u''(x_i^n)+\mathcal O((h_-^+)^2)-u'(x_i^n)'+\frac{h_i^-}{2}u''(x_i^n)+\mathcal O((h_i^-)^2)}{h_i^-+h_i^+}\nonumber \\
       &=\frac{1}{2}u''(x_i^n)+\frac{O((h_i^-)^2)+O((h_i^+)^2)}{h_i^-+h_i^+}\label{DiscrSecondDrv}
\end{align}
   
We perform a mesh reconstruction and a solution update step which yield the new mesh $M_x^{n+1}$ and the update solution $\hat U^n$. We define
$$h_1=x_i^n-x_{j}^{n+1},\quad h_2=x_{j+1}^{n+1}-x_i^n$$
and note that due to the linearity and the interpolation the following is valid
\begin{align}
    \frac{u_i^n-\hat u_j^n}{h_1}=\lambda_{i-1}^n&\Rightarrow \hat u_j^n=u_i^n-\lambda_{i-1}^n h_1\nonumber\\
    \frac{\hat u_{j+1}^n-u_i^n}{h_2}=\lambda_i^n&\Rightarrow \hat u_{j+1}^n=u_i^n+\lambda_i^n h_2\nonumber
\end{align}
We interpolate linearly over the points $(x_j^{n+1},\hat u_j^n)$, $(x_{j+1}^{n+1},\hat u_{j+1}^n)$ and compute the new value $\hat {\hat u}_i^n$ of the previous node $x_i^n$ (cf Figure (\ref{Graph.Diffusion})). To do so, we note that $x_i^n\in(x_j^{n+1},x_{j+1}^{n+1})$ and is given by the following convex combination $$x_i^n=\frac{h_2}{h_1+h_2} x_j^{n+1}+\frac{h_1}{h_1+h_2} x_{j+1}^{n+1},$$
and so the projected value $\hat {\hat u}_i^n$ of the node $x_i^n$ over the new linear segment is given by the same convex combination with respect to $\hat u_j^n$ and  $\hat u_{j+1}^n$,
\begin{align}
  \hat {\hat u}_i^n&=\frac{h_2}{h_1+h_2}\hat u_j^n+\frac{h_1}{h_1+h_2}\hat u_{j+1}^n\nonumber\\
                   &=\frac{h_2}{h_1+h_2}(u_i^n-\lambda_{i-1}^nh_1)+\frac{h_1}{h_1+h_2}(u_i^n+\lambda_i^nh_2)\nonumber\\
                   &=u_i^n+\frac{h_1h_2}{h_1+h_2}(\lambda_i^n-\lambda_{i-1}^n)\label{DiscreteRelStar}\\
                   &=u_i^n+\frac{C}{2}u''(x_i^n)+\frac{h_1h_2}{h_1+h_2}\left(O((h_i^-)^2)+O((h_i^+)^2)\right)\quad 
                       \mbox{ (using Rel.(\ref{DiscrSecondDrv}))}\label{Diffusion.coeff}
\end{align}
where $C=\frac{h_1h_2}{h_1+h_2}(h_i^-+h_i^+)\geq 0$. 
\begin{remark}
  Rel. (\ref{Diffusion.coeff}) reveals the diffusive property of the mesh reconstruction (Step 1.) and the solution update (Step 2.). We moreover note that in 
  the case of a uniform mesh either $h_1=0$ or $h_2=0$, which both yield $C=0$.
\end{remark}
As explained earlier, this remark is useful only if there are several nodes in between two consequent extremes. In this case, continuous repetitions of the mesh reconstruction and the solution update  guarantee that the extremes will satisfy the $\lambda$-rule effect. We note though that in practise, applying the mesh reconstruction procedure and the solution update once in every time step is sufficient.

\section{Numerical tests}\label{Section.Numerics}
As stated at the introduction the numerical schemes that we shall discuss are oscillatory, either due to their dispersive or to their anti-diffusive nature. Here we shall only state their description for non-uniform meshes and prove that they satisfy the Evolution Requirement Req.(\ref{EvolutionReq}), which we restate
\begin{requirement*}[Evolution requirement]
  There exists a constant $C>0$ independent of the time step $n$ and the node $i$ such that,
  \begin{equation}
    |u_i^{n+1}-\hat u_i^n|\leq C\max\left\{|\hat u_{i+1}^n-\hat u_i^n|,|\hat u_i^n-\hat u_{i-1}^n|\right\}
  \end{equation}
\end{requirement*}

The problems that we shall deal with, are the Transport Equation
$$u_t+u_x=0,\quad x\in[0,1],\ t\in[0,1]$$
and the inviscid Burgers Equation 
$$u_t+\left(\frac{u^2}{2}\right)_x=0,\quad x\in[0,1],\ t\in[0,1]$$
both with jump initial conditions
$$u_0(x)=\mathcal X_{[0,1/2]}(x),\quad x\in[0,1]$$ 

\subsection{Richtmyer 2-step Lax-Wendroff}
  In this approach we consider the non-uniform cell centered discretization of the domain in cells 
  $$C_i=(x_{i-1/2}^n,x_{i+1/2}^n)\quad\mbox{ with }\quad |C_i^n|=h_i^n$$
  The mesh $M_x^n=\{x_i^n, i\in \Z\}$ consists of the middle points, 
  $$x_i^n=\frac{x_{i+1/2}^n+x_{i-1/2}^n}{2}\quad \mbox{ hence }\quad x_i^n-x_{i-1}^n=\frac{h_i^n+h_{i-1}^n}{2}$$
  For this description of the grid we propose the following numerical scheme as the generalisation on non-uniform meshes of the Richtmyer 2-step Lax-Wendorff numerical 
  scheme,

  \begin{align}
    u^\ast_{i+1/2}&=\frac{h_{i+1}^{n+1}}{h_i^{n+1}+h_{i+1}^{n+1}}\hat u_i^n+\frac{h_{i}^{n+1}}{h_i^{n+1}+h_{i+1}^{n+1}}\hat u_{i+1}^n
                   -\frac{\Delta t}{h_i^{n+1}+h_{i+1}^{n+1}}(f(\hat u_{i+1}^n)-f(\hat u_i^n))\nonumber \\
    u_i^{n+1}&=u_i^n-\frac{\Delta t}{h_i}\big(f(u_{i+1/2}^\ast)-f(\hat u_{i-1/2}^\ast)\big)\nonumber
  \end{align}
  or 
  $$u_i^{n+1}=\hat u_i^n-\frac{\Delta t}{h_i^{n+1}}\big(F_{i+1/2}-F_{i-1/2}\big)$$
  with 
  $$F_{i+1/2}=f(u_{i+1/2}^\ast)=f\Big(\frac{h_{i+1}^{n+1}}{h_i^{n+1}+h_{i+1}^{n+1}}\hat u_i^n+\frac{h_{i}^{n+1}}{h_i^{n+1}+h_{i+1}^{n+1}}\hat u_{i+1}^n
                               -\frac{\Delta t}{h_i^{n+1}+h_{i+1}^{n+1}}(f(\hat u_{i+1}^n)-f(\hat u_i^n))\Big)$$ 
  \begin{remark}
    A straight forward computation reveals that this scheme reduces to the usual Richtmyer 2-step Lax-Wendroff scheme when then mesh is uniform i.e $h_i^n=h$ for 
    every $i,n$.
  \end{remark}
  We need to bound the difference
  \begin{align}
    |u_i^{n+1}-u_i^n| 
        &\leq \frac{\Delta t}{h_i}\left |f(u_{i+1/2}^\ast)-f(u_{i-1/2}^\ast) \right|\nonumber\\
        &\leq \frac{\Delta t}{h_i}\max |f'| |u_{i+1/2}^\ast-u_{i-1/2}^\ast|\nonumber\\
        &\leq CFL |u_{i+1/2}^\ast-u_{i-1/2}^\ast|\nonumber
  \end{align}
  and now the difference
  \begin{align}
      |u_{i+1/2}^\ast-u_{i-1/2}^\ast|
        &=\bigg| \frac{h_{i+1}^{n+1}}{h_i^{n+1}+h_{i+1}^{n+1}}\hat u_i^n+\frac{h_i^{n+1}}{h_i^{n+1}+h_{i+1}^{n+1}}\hat u_{i+1}^n
                -\frac{\Delta t}{h_i^{n+1}+h_{i+1}^{n+1}}\left(f(\hat u_{i+1}^n)-f(\hat u_i^n)\right) \nonumber\\
        &\quad -\frac{h_i^{n+1}}{h_{i-1}^{n+1}+h_i^{n+1}}\hat u_{i-1}^n -\frac{h_{i-1}^{n+1}}{h_{i-1}^{n+1}+h_i^{n+1}}\hat u_i^n
               +\frac{\Delta t}{h_{i-1}^{n+1}+h_i^{n+1}}\left(f(\hat u_i^n)-f(\hat u_{i-1}^n)\right)\bigg|\nonumber
  \end{align}
  \begin{figure}[t]
  \begin{tabular}{cc}
    \includegraphics[width=7cm]{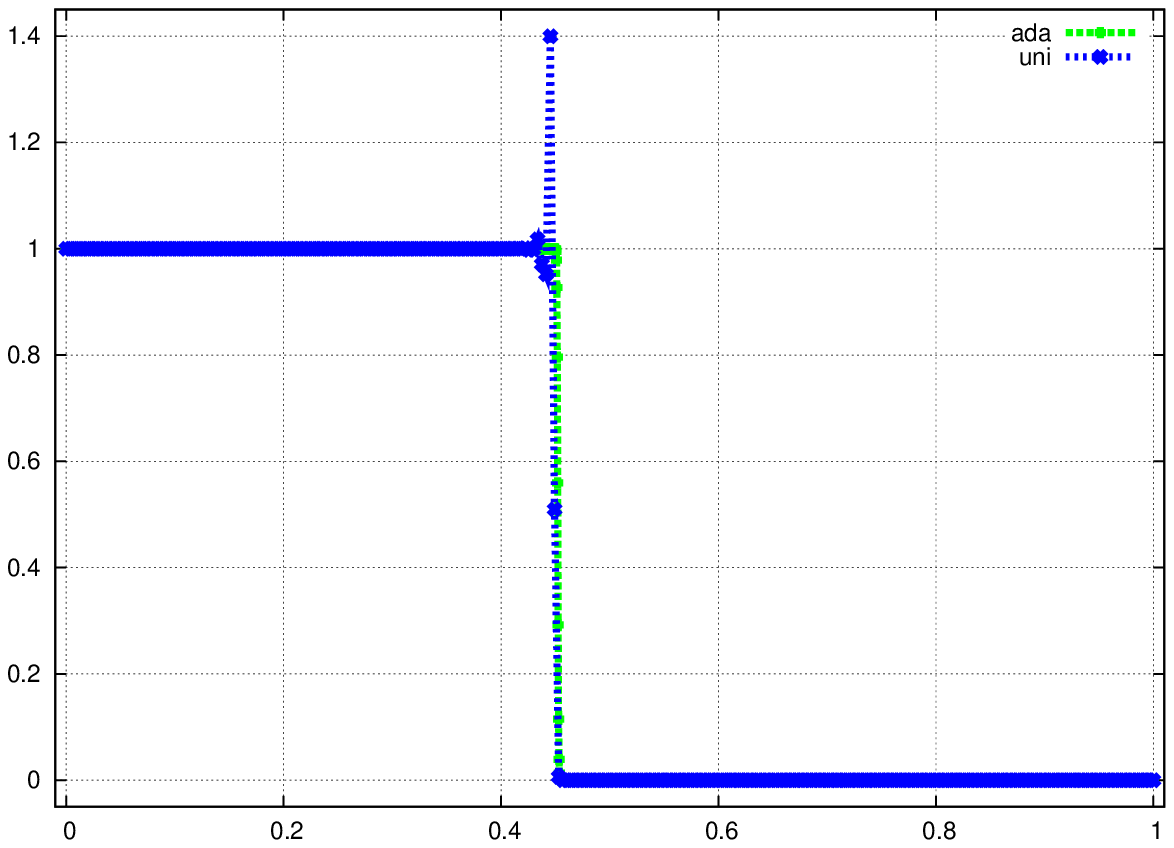}&  \includegraphics[width=7cm]{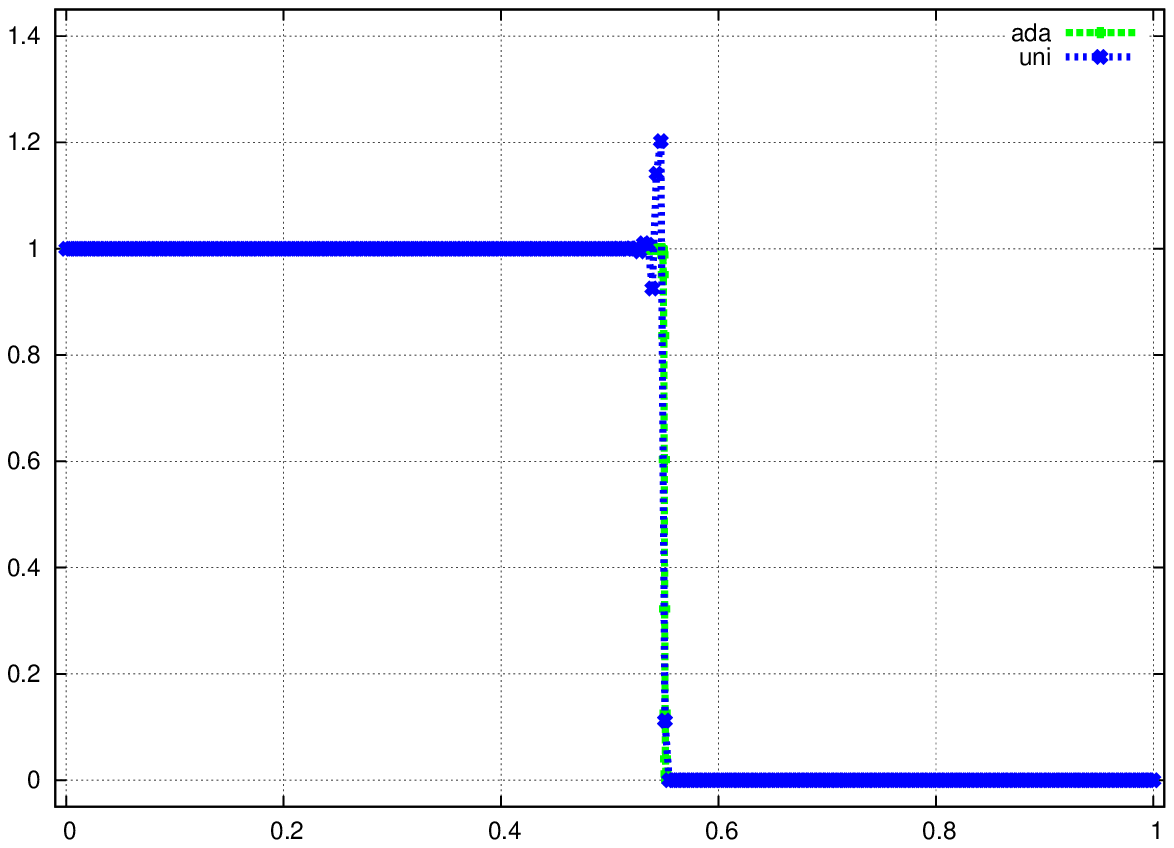}\cr
    \includegraphics[width=7cm]{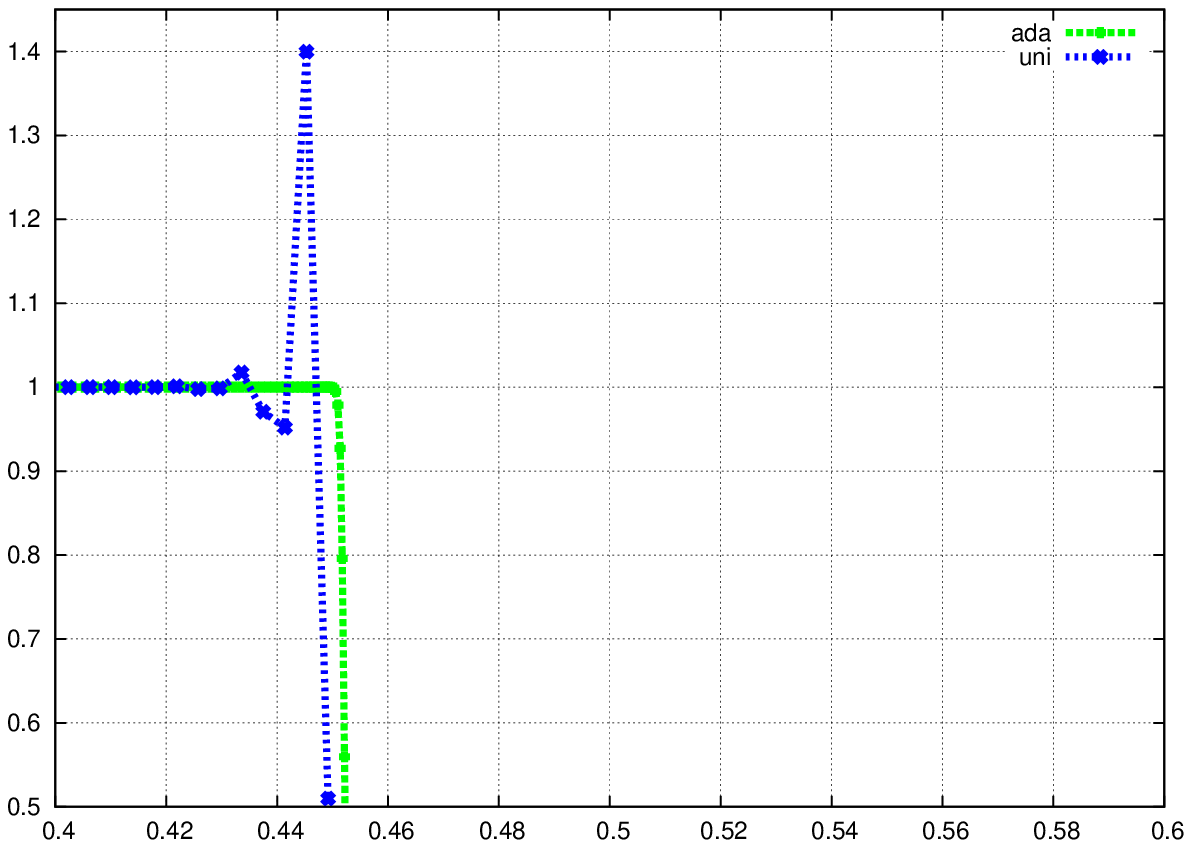}& \includegraphics[width=7cm]{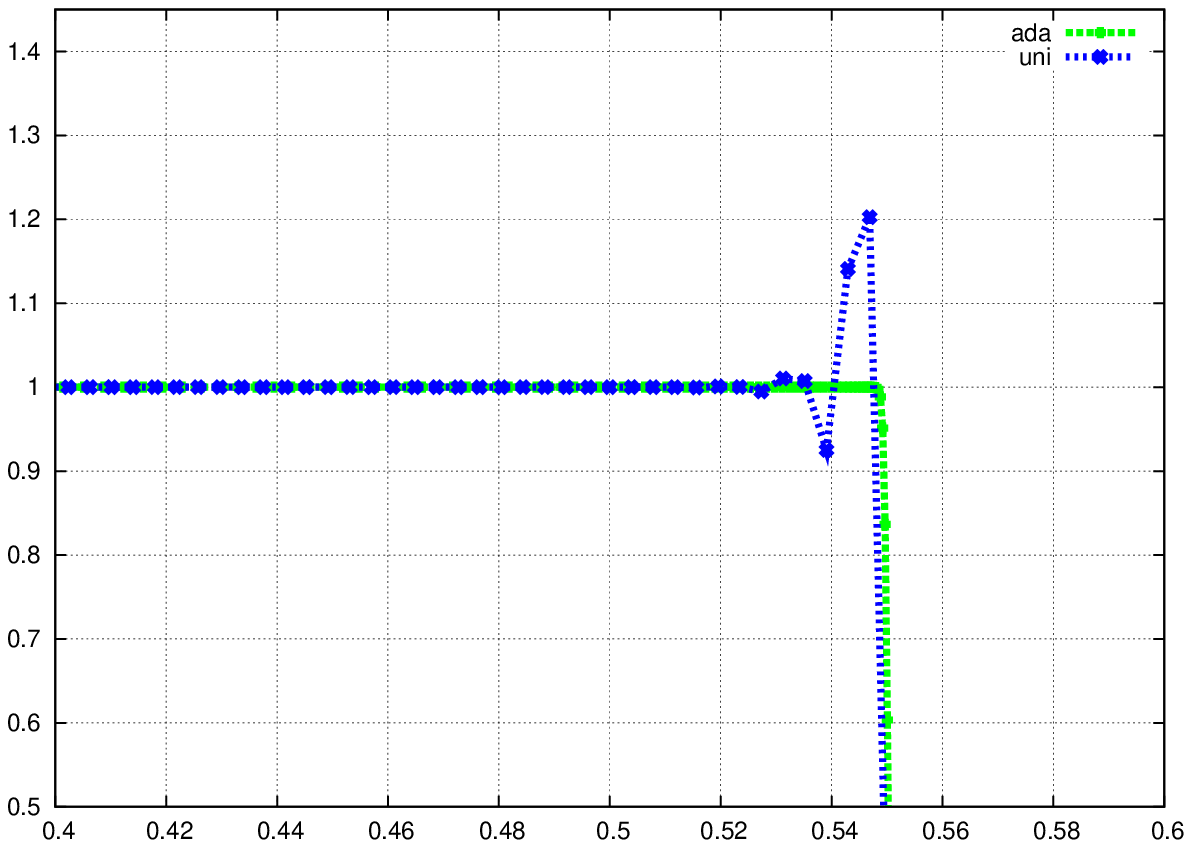}\cr
  \end{tabular}
  \caption{In this graph are presented two time instances of the numerical solutions of the inviscid Burgers equation using the Richtmyer 2-step Lax Wendroff over  
           uniform and non-uniform meshes. In the first row of graphs the full domain is presented, while in the second row the graphs are focused on the top of the 
           shock. The uniform mesh case -depicted in blue- exhibits oscillations due to the dispersive nature of the numerical scheme where as the non-uniform mesh 
           case -depicted in green- is clean.
          }\label{Graph.LxW.Burgers}
  \end{figure}
    by setting $\frac{h_{i+1}^{n+1}}{h_i^{n+1}+h_{i+1}^{n+1}}=\mu_1$ and $\frac{h_i^{n+1}}{h_{i-1}^{n+1}+h_i^{n+1}}=\mu_2$ the previous bound recasts into:
    \begin{align}
      |u_{i+1/2}^\ast-u_{i-1/2}^\ast|
        &=\bigg| \mu_1\hat u_i^n+(1-\mu_1)\hat u_{i+1}^n-\frac{\Delta t}{h_i^{n+1}+h_{i+1}^{n+1}}\left(f(\hat u_{i+1}^n)-f(\hat u_i^n)\right) \nonumber\\
        &\qquad -\mu_2\hat u_{i-1}^n-(1-\mu_2)\hat u_i^n+\frac{\Delta t}{h_{i-1}^{n+1}+h_i^{n+1}}\left(f(\hat u_i^n)-f(\hat u_{i-1}^n)\right)\bigg|\nonumber\\
        &\leq \mu_1|\hat u_i^n-\hat u_{i+1}^n|+\mu_2|\hat u_i^n-\hat u_{i-1}^n|+|\hat u_{i+1}^n-\hat u_i^n|\nonumber\\
        &\qquad +\frac{\Delta t}{2\min h_i^{n+1}}\max|f'||\hat u_{i+1}^n-\hat u_i^n|+\frac{\Delta t}{2\min h_i^{n+1}}\max|f'||\hat u_i^n-\hat u_{i-1}^n|\nonumber\\
        &\leq (1+\mu_1+\mu_2+CFL)\max\{|\hat u_{i-1}^n-\hat u_i^n|,|\hat u_i^n-\hat u_{i+1}^n|\}\nonumber\\
        &\leq (3+CFL)\max\{|\hat u_{i-1}^n-\hat u_i^n|,|\hat u_i^n-\hat u_{i+1}^n|\}\nonumber
    \end{align}
    \begin{figure}[t]
    \begin{tabular}{cc}
      \includegraphics[width=7cm]{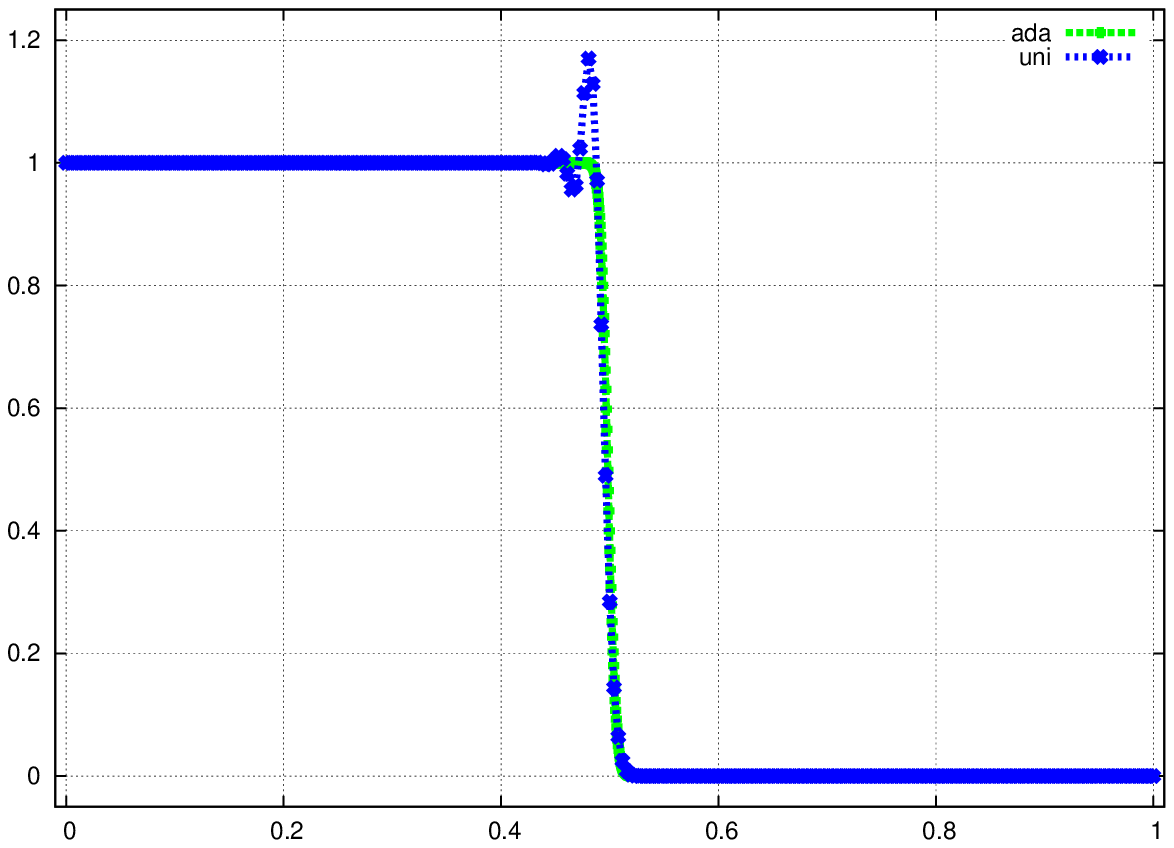}& \includegraphics[width=7cm]{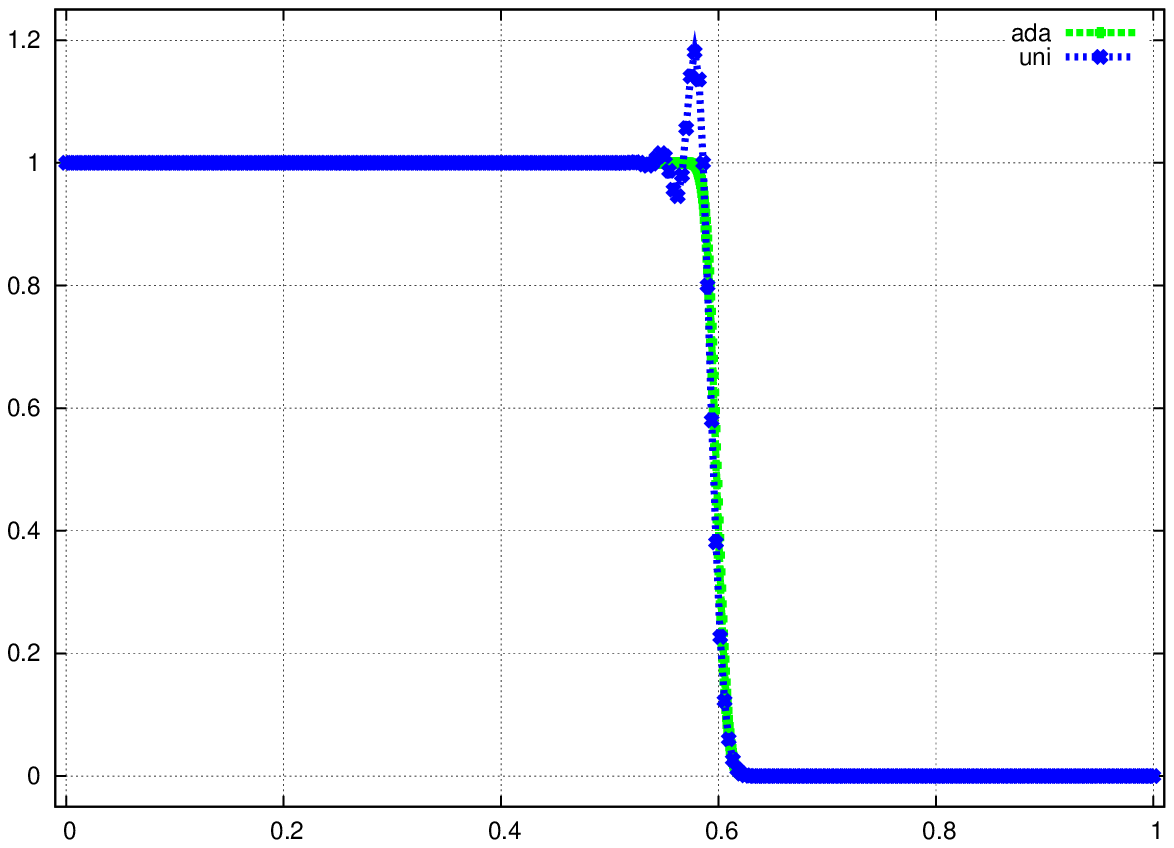}\cr
      \includegraphics[width=7cm]{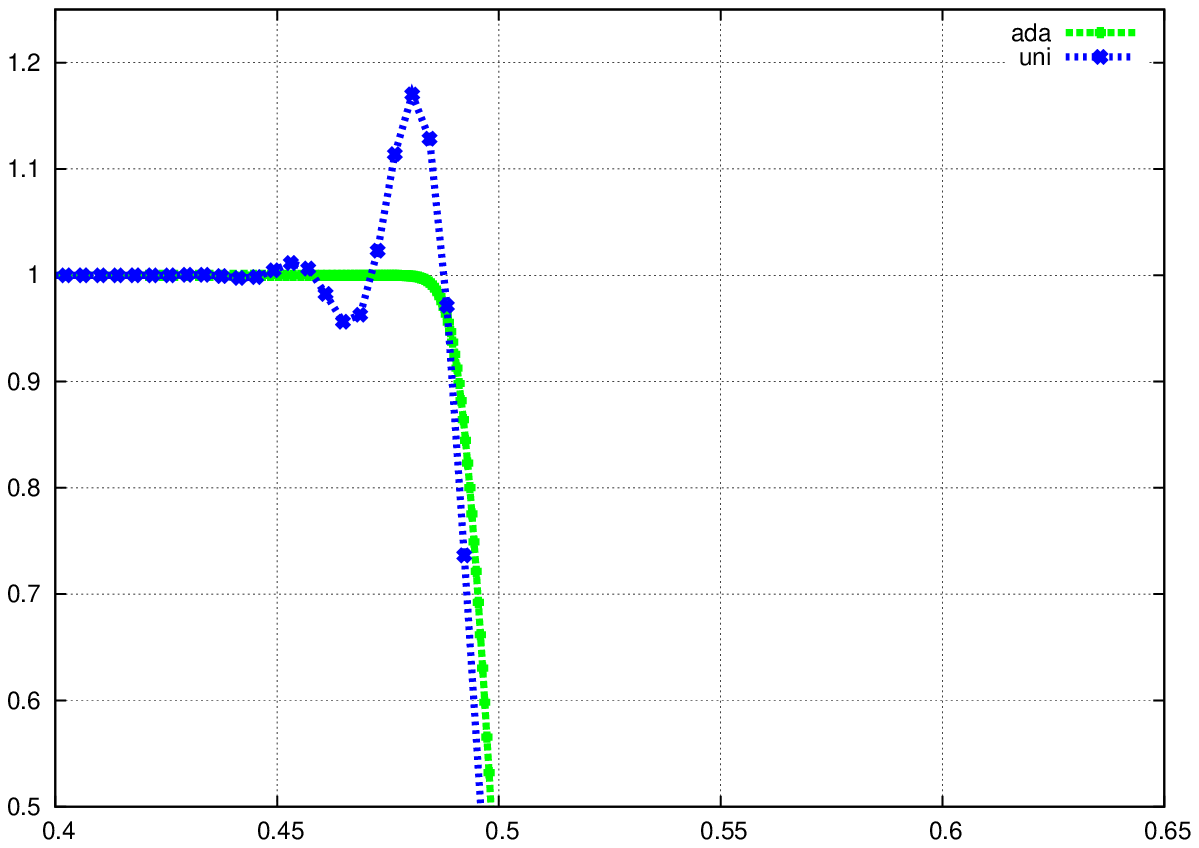}& \includegraphics[width=7cm]{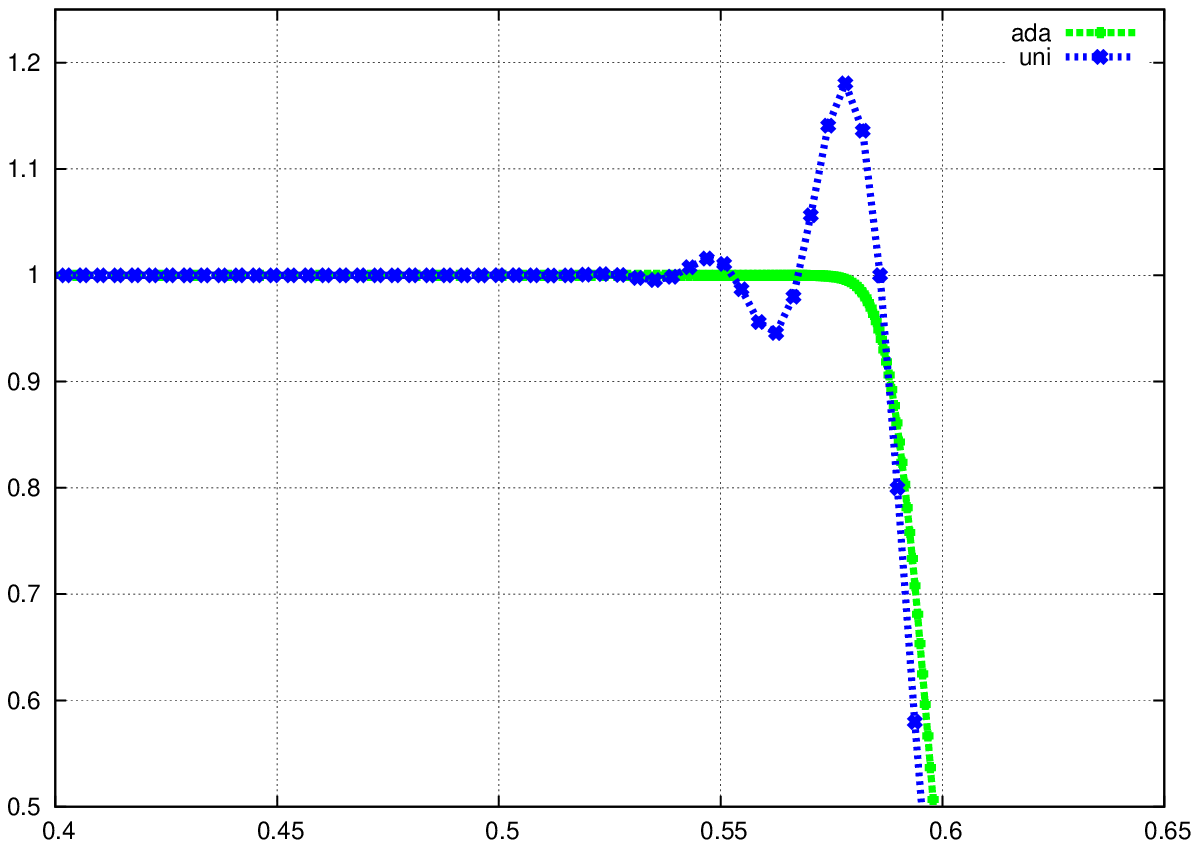}\cr
    \end{tabular}
    \caption{Transport equation with velocity $a=1$, using the Richtmyer 2-step Lax Wendroff. Again, oscillations appear in the uniform mesh case, whereas the  
             non-uniform one is clean.}\label{Graph.LxW.Transport}
    \end{figure}

    where the last inequality is valid since $0<\mu_1,\mu_2\leq 1$. So the overall bound reads,
    \begin{align}
      |u_i^{n+1}-\hat u_i^n|\leq CFL(3+CFL)\max\left\{|\hat u_{i+1}^n-\hat u_i^n|,|\hat u_i^n-\hat u_{i-1}^n|\right\}\nonumber
    \end{align}
    The constant $C$ in this case is chosen to be $C=CFL(3+CFL)$, for this choice the evolution requirement is satisfied. We refer to Figures \ref{Graph.LxW.Burgers} 
    and \ref{Graph.LxW.Transport} for comparative graphs between the uniform and non-uniform mesh case for the Richtmyer 2-step Lax-Wendroff scheme. 
\subsection{MacCormack}
  In this approach we consider the non-uniform cell centered discretization of the domain in cells 
  $$C_i^n=(x_{i-1/2}^n,x_{i+1/2}^n)\quad\mbox{ with }\quad |C_i^n|=h_i^n$$
  The mesh $M_x^n=\{x_i^n, i\in \Z\}$ consists of the middle points, 
  $$x_i^n=\frac{x_{i+1/2}^n+x_{i-1/2}^n}{2}\quad \mbox{ hence }\quad x_i^n-x_{i-1}^n=\frac{h_i^n+h_{i-1}^n}{2}$$
  For this description of the grid we propose the following scheme as the generalisation of the MacCormack,
  \begin{align}
           u_i^\ast&=\hat u_i^n-\frac{2\Delta t}{h_i^{n+1}+h_{i+1}^{n+1}}(f(\hat u_{i+1}^n)-f(\hat u_i^n))\nonumber\\
     u_i^{\ast\ast}&=u_i^\ast-\frac{2\Delta t}{h_{i-1}+h_i}(f(u_i^\ast)-f(u_{i-1}^\ast))\nonumber\\
          u_i^{n+1}&=\frac{\hat u_i^n+u_i^{\ast\ast}}{2}\nonumber
  \end{align}
  We first rewrite the scheme in the following form, for $f_i^\ast=f(u_i^\ast)$ and $f_i=f(\hat u_i^n)$
  \begin{align}
    u_i^{n+1}&=\frac{\hat u_i^n}{2}+\frac{u_i^\ast-\frac{2 \Delta t}{h_{i-1}+h_i}(f_i^\ast-f_{i-1}^\ast)}{2}\nonumber\\
             &=\frac{\hat u_i^n}{2}+\frac{\hat u_i^n-\frac{2 \Delta t}{h_i^{n+1}+h_{i+1}^{n+1}}(f_{i+1}-f_i)
                                                    -\frac{2 \Delta t}{h_{i-1}^{n+1}+h_i^{n+1}}(f_i^\ast-f_{i-1}^\ast)}{2}\nonumber\\
             &=\hat u_i^n-\frac{\Delta t}{h_i^{n+1}+h_{i+1}^{n+1}}(f_{i+1}-f_i)-\frac{\Delta t}{h_{i-1}^{n+1}+h_i^{n+1}}(f_i^\ast-f_{i-1}^\ast)\nonumber
  \end{align}
  \begin{figure}[t]
    \begin{tabular}{cc}
      \includegraphics[width=7cm]{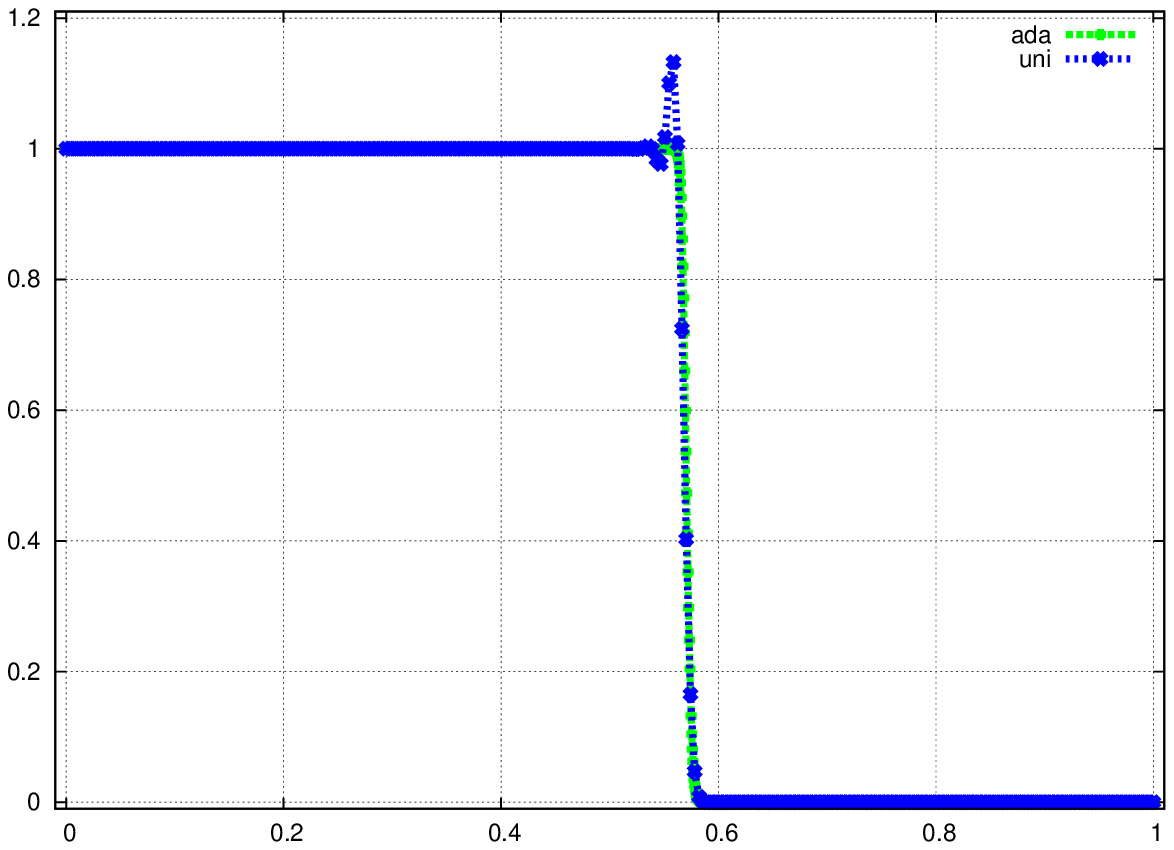}&  \includegraphics[width=7cm]{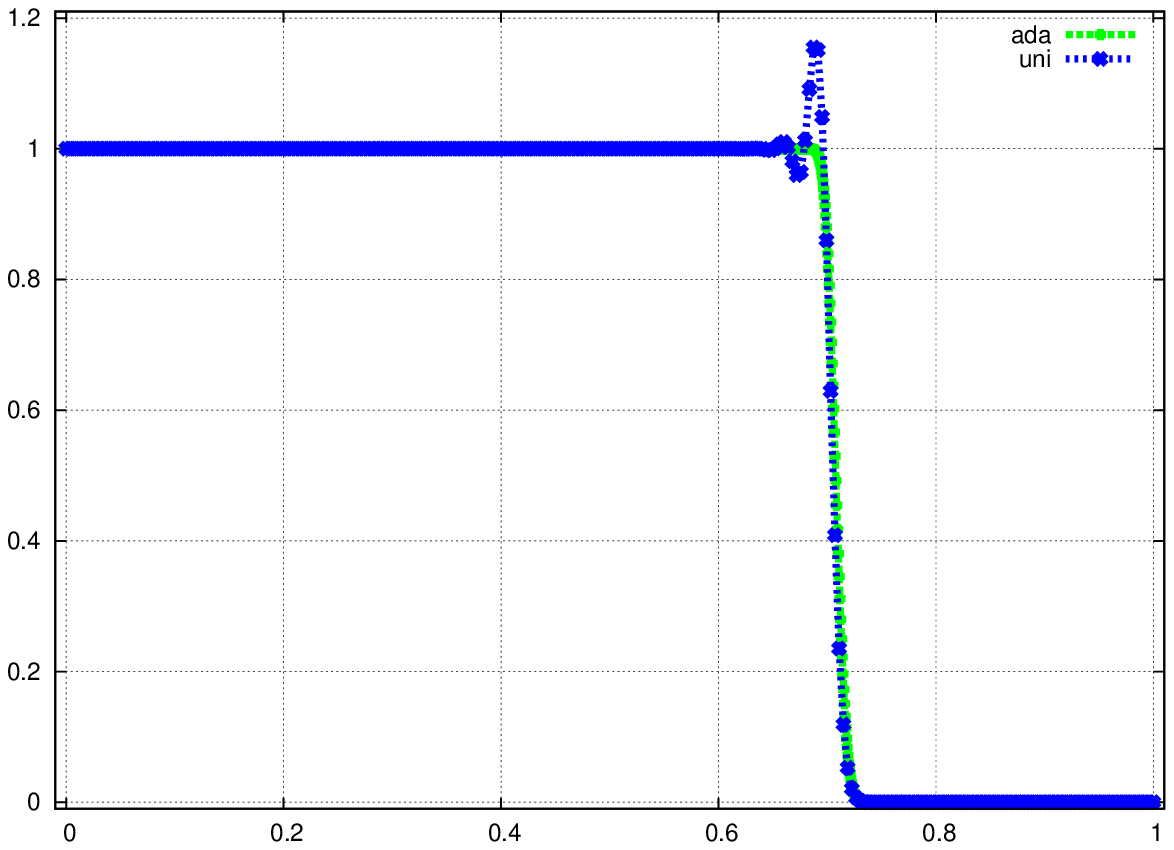}\cr
      \includegraphics[width=7cm]{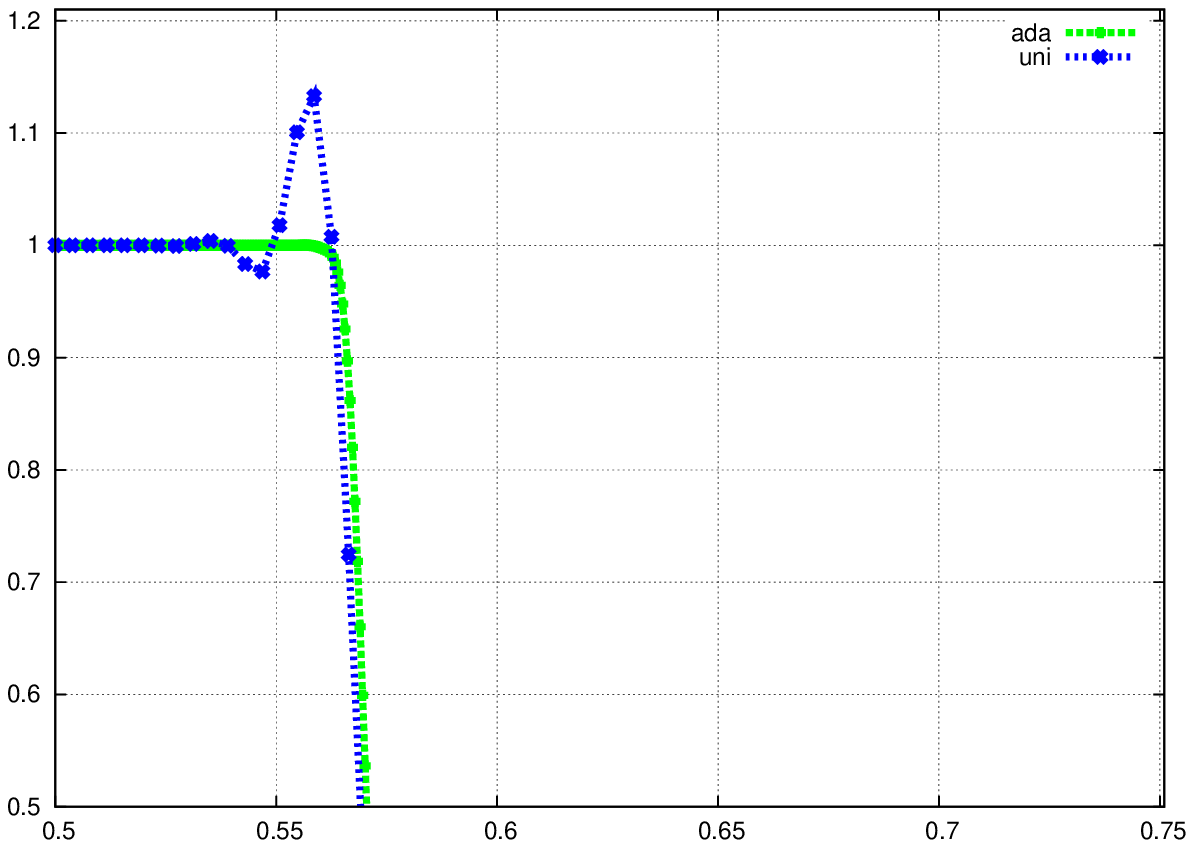}&
                                  \includegraphics[width=7cm]{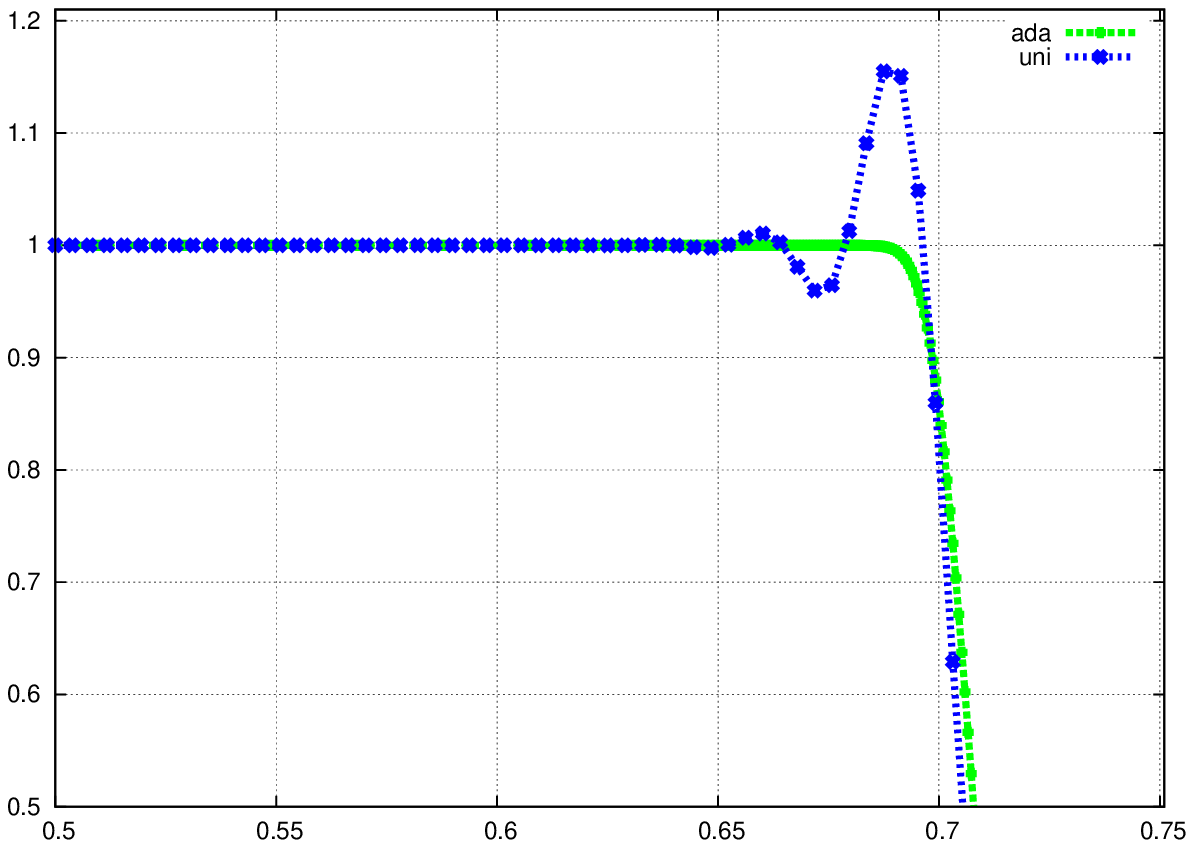}\cr
    \end{tabular}
    \caption{Transport equation with velocity $a=1$, using the MacCormack scheme over uniform and non-uniform meshes. Again oscillations are apparent in the uniform 
             mesh case -due to the dispersive nature of the scheme- whereas the non-uniform is clean.}\label{Graph.Mac.Transport}
  \end{figure}
  So, to prove the Evolution Requirement  for this scheme we need to bound
  $$|u_i^{n+1}-\hat u_i^n|=|-\frac{\Delta t}{h_i^{n+1}+h_{i+1}^{n+1}}(f_{i+1}-f_i)-\frac{\Delta t}{h_{i-1}^{n+1}+h_i^{n+1}}(f_i^\ast-f_{i-1}^\ast)|$$
  which reads,
  \begin{align}
    |u_i^{n+1}-\hat u_i^n|
      &\leq \frac{CFL}{2}\left(|\hat u_{i+1}^n-\hat u_i^n|+|u_i^\ast-u_{i-1}^\ast|\right)\nonumber \\
      &\leq \frac{CFL}{2}\left(|\hat u_{i+1}^n-\hat u_i^n|+|\hat u_i^n-\hat u_{i-1}^n|+CFL|\hat u_i^n-\hat u_{i-1}^n|+CFL|\hat u_{i+1}^n-\hat u_i^n|\right)\nonumber\\
      &\leq CFL(1+CFL)\max\left\{|\hat u_{i+1}^n-\hat u_i^n|,|\hat u_i^n-\hat u_{i-1}^n|\right\}\nonumber
  \end{align}
  So, the constant $C$ in this case is chosen to be $C=CFL(1+CFL)$ and for this choice the Evolution Requirement is satisfied. We refer to Figure 
  \ref{Graph.Mac.Transport} for a comparison graph between the uniform and non-uniform mesh case for the MacCormack scheme.
\subsection{Unstable Centered - FTCS}
  This scheme produces oscillations due to its anti-diffusive nature, which property is also responsible for the instability of the scheme. 

  In this approach we consider the non-uniform mesh 
  $$M_x^n=\{x_i^n, i\in \Z\}\quad\mbox{ with } h_i^n=x_i^n-x_{i-1}^n$$ 
  The middle points $x_{i-1/2}^n=\frac{x_{i-1}^n+x_i^n}{2}$ define a partition of the domain in cells, 
  $$C_i^n=(x_{i-1/2}^n,x_{i+1/2}^n)\quad\mbox{ with }\quad |C_i^n|=\frac{h_i^n+h_{i+1}^n}{2}$$
  \begin{figure}[t]
    \begin{tabular}{cc}
      \includegraphics[width=7cm]{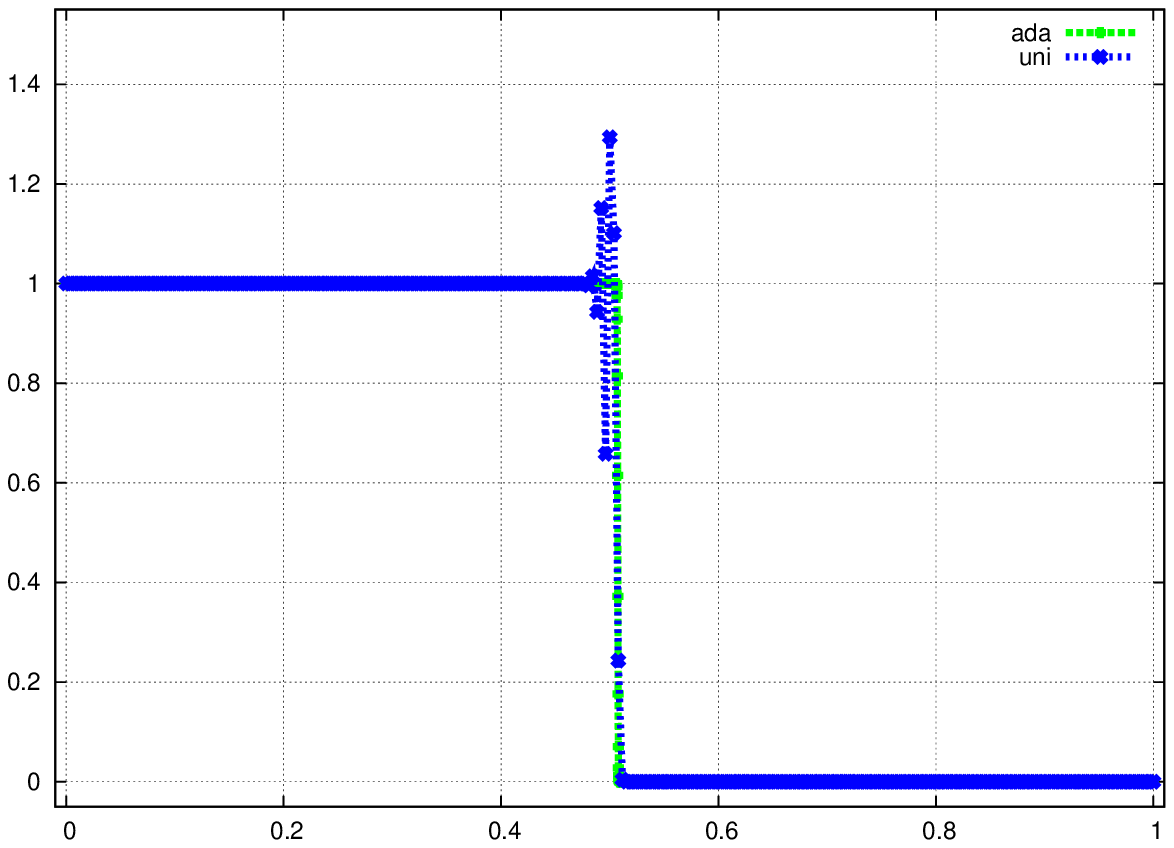}&  \includegraphics[width=7cm]{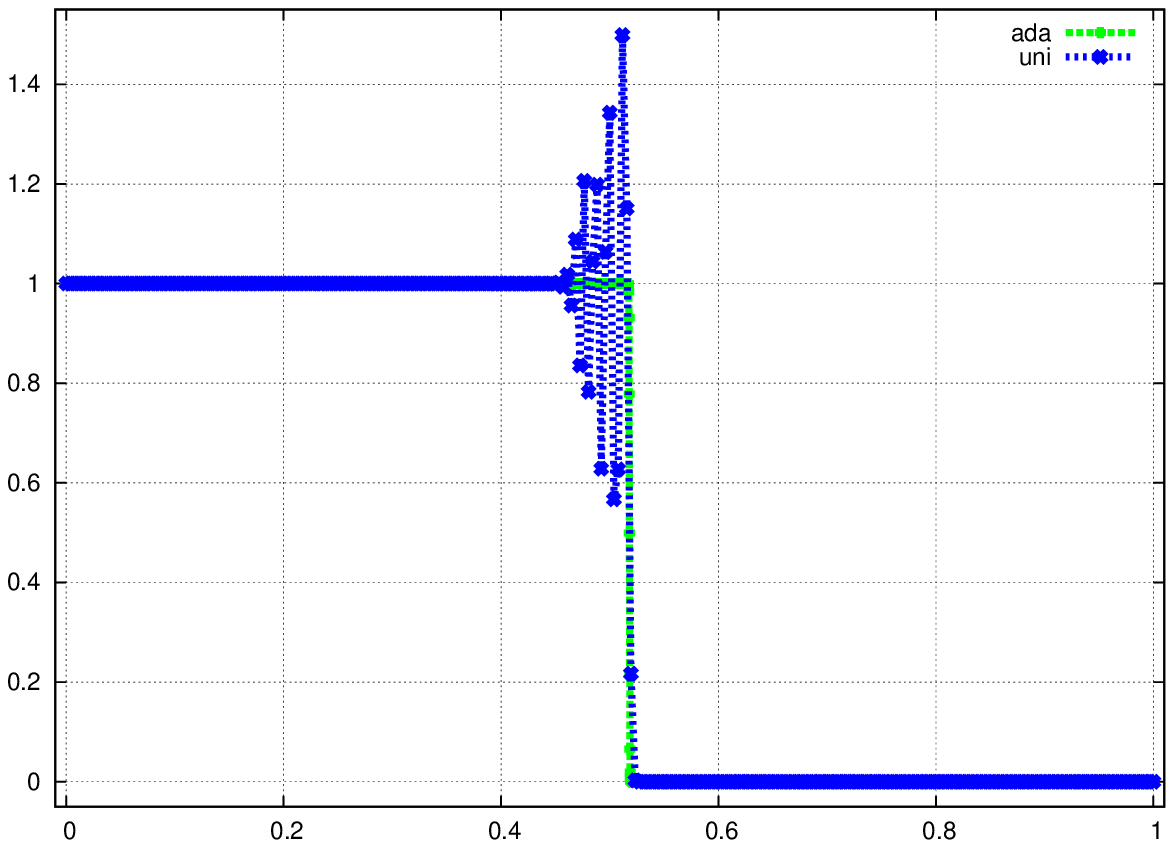}\cr
      \includegraphics[width=7cm]{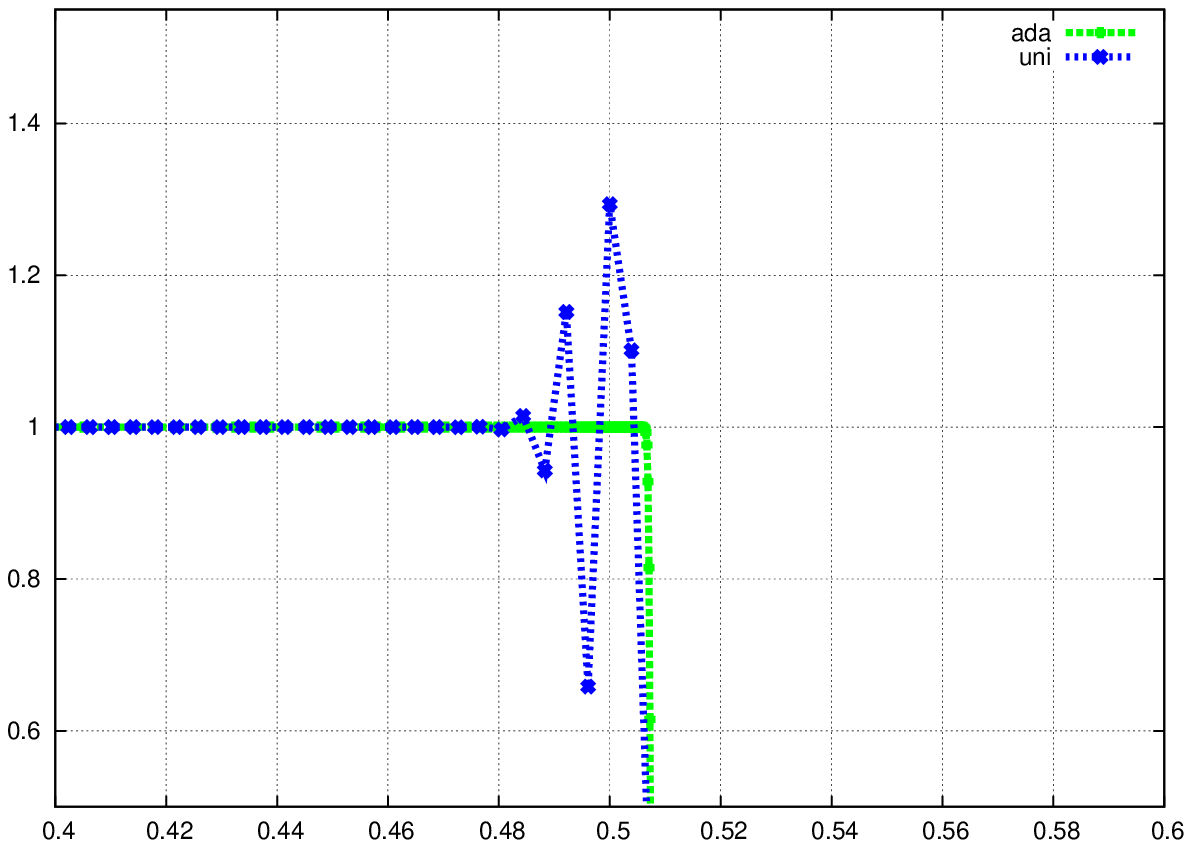}& \includegraphics[width=7cm]{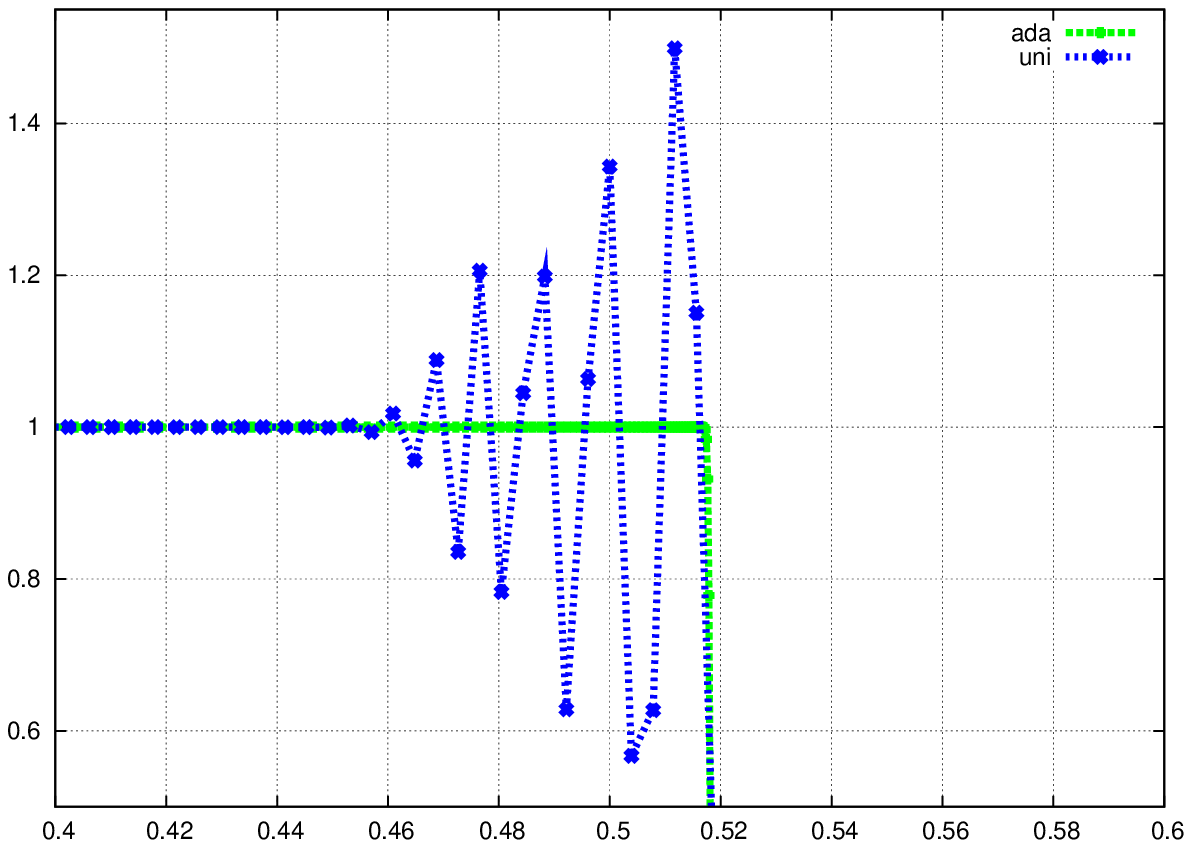}\cr
    \end{tabular}
    \caption{Inviscid Burgers equation, using the unstable FTCS. The oscillations in the uniform case are due to the anti-diffusive nature of the numerical scheme. The 
             non-uniform mesh case is clean.}\label{Graph.FTCS.Burgers}
  \end{figure}

  For this description of the grid we discuss the known to be \emph{unstable} Forward in Time Centered in Space (FTCS) scheme
  $$u_i^{n+1}=\hat u_i^n-\frac{\Delta t}{h_i^{n+1}+h_{i+1}^{n+1}}\big(f(\hat u_{i+1}^n)-f(\hat u_{i-1}^n)\big)$$
  This scheme can be written in conservative form as follows,
  $$u_i^{n+1}=\hat u_i^n-\frac{2\Delta t}{h_i^{n+1}+h_{i+1}^{n+1}}(F_{i+1/2}^n-F_{i-1/2}^n)$$
  with 
  $$F_{i+1/2}^n=\frac{f(\hat u_i^n)+f(\hat u_{i+1}^n)}{2}$$

  Easily we deduce that
  \begin{align}
    |u_i^{n+1}-\hat u_i^n|
      &\leq \frac{\Delta t}{2\min h_i^{n+1}}\max|f'||\hat u_{i+1}^n-\hat u_{i-1}^n|
       \leq \frac{CFL}{2}\left(|\hat u_{i+1}^n-\hat u_i^n|+|\hat u_i^n-\hat u_{i-1}^n|\right)\nonumber \\
      &\leq CFL \max\left\{|\hat u_{i+1}^n-\hat u_i^n|,|\hat u_i^n-\hat u_{i-1}^n|\right\}\nonumber
  \end{align}
  The constant $C$ in this case is chosen to be $C=CFL$, for this choice the Evolution requirement is satisfied. We refer to Figure \ref{Graph.FTCS.Burgers} for a 
  comparison graph between the uniform and non-uniform mesh case for the FTCS scheme.

\section{Conclusions}

In this work we investigated the creation and evolution of oscillations over non-uniform, adaptively redefined meshes. The mesh reconstruction is driven by the geometry of the numerical solution and the solution update is performed by interpolation over piecewise linear functions. The numerical schemes considered are 3-point non-uniform versions of oscillatory numerical schemes. The overall process is dictated by the Main Adaptive Scheme (MAS).

We prove under specific assumptions/requirements on the mesh reconstruction, the solution update and the numerical schemes that the numerical solution is of Bounded Total Variation; furthermore, under more strict assumptions, we prove that the increase of the Total Variation decreases with time. We also describe the diffusive property of the mesh reconstruction and solution update steps of MAS and finally we provide numerical test supporting the outcomes of this study.

\nocite{*}
\bibliographystyle{plain}
\bibliography{Bibliography}

\end{document}